\PassOptionsToPackage{unicode}{hyperref}
\PassOptionsToPackage{hyphens}{url}
\PassOptionsToPackage{dvipsnames,svgnames,x11names}{xcolor}
\documentclass[
  11pt]{article}

\usepackage{amsmath,amssymb}
\usepackage{amsthm}
\usepackage{amsmath,amssymb,amsfonts,amsthm,subfloat,algorithm,url,float,epsf,psfrag,bm}
\usepackage{mathtools}
\usepackage[pdftex]{color,graphicx}
\usepackage[colorlinks,citecolor=blue,urlcolor=blue]{hyperref}
\usepackage[noabbrev,capitalize]{cleveref}
\usepackage{algorithm, algpseudocode, verbatim, float}
\usepackage{marvosym}
\usepackage{titlesec}
\usepackage{xr}
\usepackage[pagewise]{lineno}
\usepackage{subcaption}
\usepackage[utf8]{inputenc}
\usepackage{mathrsfs}
\usepackage{amsfonts}
\usepackage{enumerate}
\usepackage{latexsym}
\usepackage{multirow}

\allowdisplaybreaks

\usepackage{yhmath}
\usepackage{iftex}
\ifPDFTeX
  \usepackage[T1]{fontenc}
  \usepackage[utf8]{inputenc}
  \usepackage{textcomp} 
\else 
  \usepackage{unicode-math}
  \defaultfontfeatures{Scale=MatchLowercase}
  \defaultfontfeatures[\rmfamily]{Ligatures=TeX,Scale=1}
\fi
\usepackage{lmodern}
\ifPDFTeX\else  
\fi
\IfFileExists{upquote.sty}{\usepackage{upquote}}{}
\IfFileExists{microtype.sty}{
  \usepackage[]{microtype}
  \UseMicrotypeSet[protrusion]{basicmath} 
}{}
\makeatletter
\@ifundefined{KOMAClassName}{
  \IfFileExists{parskip.sty}{%
    \usepackage{parskip}
  }{
    \setlength{\parindent}{0pt}
    \setlength{\parskip}{6pt plus 2pt minus 1pt}}
}{
  \KOMAoptions{parskip=half}}
\makeatother
\usepackage{xcolor}
\makeatletter
\ifx\paragraph\undefined\else
  \let\oldparagraph\paragraph
  \renewcommand{\paragraph}{
    \@ifstar
      \xxxParagraphStar
      \xxxParagraphNoStar
  }
  \newcommand{\xxxParagraphStar}[1]{\oldparagraph*{#1}\mbox{}}
  \newcommand{\xxxParagraphNoStar}[1]{\oldparagraph{#1}\mbox{}}
\fi
\ifx\subparagraph\undefined\else
  \let\oldsubparagraph\subparagraph
  \renewcommand{\subparagraph}{
    \@ifstar
      \xxxSubParagraphStar
      \xxxSubParagraphNoStar
  }
  \newcommand{\xxxSubParagraphStar}[1]{\oldsubparagraph*{#1}\mbox{}}
  \newcommand{\xxxSubParagraphNoStar}[1]{\oldsubparagraph{#1}\mbox{}}
\fi
\makeatother

\usepackage{longtable,booktabs,array}
\usepackage{calc} 
\usepackage{etoolbox}
\makeatletter
\patchcmd\longtable{\par}{\if@noskipsec\mbox{}\fi\par}{}{}
\makeatother
\IfFileExists{footnotehyper.sty}{\usepackage{footnotehyper}}{\usepackage{footnote}}
\makesavenoteenv{longtable}
\usepackage{graphicx}
\makeatletter
\def\maxwidth{\ifdim\Gin@nat@width>\linewidth\linewidth\else\Gin@nat@width\fi}
\def\maxheight{\ifdim\Gin@nat@height>\textheight\textheight\else\Gin@nat@height\fi}
\makeatother
\setkeys{Gin}{width=\maxwidth,height=\maxheight,keepaspectratio}
\makeatletter
\def\fps@figure{htbp}
\makeatother

\makeatletter
\@ifpackageloaded{caption}{}{\usepackage{caption}}
\AtBeginDocument{%
\ifdefined\contentsname
  \renewcommand*\contentsname{Table of contents}
\else
  \newcommand\contentsname{Table of contents}
\fi
\ifdefined\listfigurename
  \renewcommand*\listfigurename{List of Figures}
\else
  \newcommand\listfigurename{List of Figures}
\fi
\ifdefined\listtablename
  \renewcommand*\listtablename{List of Tables}
\else
  \newcommand\listtablename{List of Tables}
\fi
\ifdefined\figurename
  \renewcommand*\figurename{Figure}
\else
  \newcommand\figurename{Figure}
\fi
\ifdefined\tablename
  \renewcommand*\tablename{Table}
\else
  \newcommand\tablename{Table}
\fi
}
\@ifpackageloaded{float}{}{\usepackage{float}}
\floatstyle{ruled}
\@ifundefined{c@chapter}{\newfloat{codelisting}{h}{lop}}{\newfloat{codelisting}{h}{lop}[chapter]}
\floatname{codelisting}{Listing}

\makeatother
\makeatletter
\@ifpackageloaded{caption}{}{\usepackage{caption}}
\@ifpackageloaded{subcaption}{}{\usepackage{subcaption}}
\makeatother

\ifLuaTeX
  \usepackage{selnolig}  
\fi
\usepackage[numbers]{natbib}
\usepackage{bookmark}

\IfFileExists{xurl.sty}{\usepackage{xurl}}{} 
\hypersetup{
  pdftitle={Title},
  pdfauthor={Author 1; Author 2},
  pdfkeywords={3 to 6 keywords, that do not appear in the title},
  colorlinks=true,
  linkcolor={blue},
  filecolor={Maroon},
  citecolor={Blue},
  urlcolor={Blue},
  pdfcreator={LaTeX via pandoc}}

\newtheoremstyle{customgapstyle}
{8pt}      
{8pt}      
{\itshape}  
{}          
{\bfseries} 
{.}         
{.5em}      
{}          

\newcommand{\anon}{1}

\usepackage{bbm}
\theoremstyle{customgapstyle}
\newtheorem{theorem}{Theorem}

\newtheorem{lemma}{Lemma}

\newtheorem{remark}{Remark}

\numberwithin{equation}{section}
\numberwithin{example}{section}
\numberwithin{theorem}{section}
\numberwithin{lemma}{section}
\numberwithin{corollary}{section}
\numberwithin{prop}{section}
\numberwithin{definition}{section}
\numberwithin{remark}{section}

\def\argmin{\mathop{\rm argmin}}
\def\argmax{\mathop{\rm argmax}}
\def\1v{\mathbf 1}
\def\Var{\mbox{\textup{Var}}}

\def\E{\mathbb{E}}
\def\F{\mathbb{F}}

\def\Hellinger{\mathfrak{H}}
\def\P{\mathbb{P}}

\newcommand{\Ac}{\mathcal{A}}

\newcommand{\Fc}{\mathcal{F}}
\newcommand{\Gc}{\mathcal{G}}

\newcommand{\Ic}{\mathcal{I}}

\newcommand{\Pc}{\mathcal{P}}

\newcommand{\KS}{d_{\textup{KS}}}
\newcommand{\aprior}{H}
\newcommand{\trueprior}{{\aprior^*}}
\newcommand{\npmle}{{\widehat{\aprior}_n}}

\newcommand{\alphastar}{\alpha_*}
\newcommand*\diff{\mathop{}\!\mathrm{d}}
\newcommand{\abs}[1]{\left\vert#1\right\vert}
\newcommand{\Real}{\mathbb R}
\allowdisplaybreaks

\begin{document}

\def\spacingset#1{\renewcommand{\baselinestretch}%
{#1}\small\normalsize} \spacingset{1}


\if1\anon
{
  \title{\bf Poisson Empirical Bayes via Gamma-Smoothed Nonparametric Maximum Likelihood}
  \author{Taehyun Kim\thanks{
    E-mail: tk3036@columbia.edu}\hspace{.2cm}\\
    Department of Statistics, Columbia University}
  \maketitle
} \fi

\if0\anon
{
  \bigskip
  \bigskip
  \bigskip
  \begin{center}
    {\LARGE\bf Poisson Empirical Bayes via Gamma-Smoothed Nonparametric Maximum Likelihood}
\end{center}
  \medskip
} \fi

\bigskip

\begin{abstract}
\noindent Empirical Bayes methods are widely used for large-scale estimation and inference in the Poisson means problem. Existing results establish theoretical properties of the nonparametric maximum likelihood estimator (NPMLE) for optimal posterior mean estimation, but comparatively less is known about uncertainty quantification (i.e., construction of confidence sets). Two main challenges in constructing confidence sets for the latent parameters based on the NPMLE are its discreteness and its slow rate of prior estimation. We resolve these limitations by introducing a smooth NPMLE that models the prior as a Gamma mixture, which is a flexible class capable of approximating a wide range of continuous priors on $(0,\infty)$. This procedure preserves the convex optimization structure of the classical NPMLE. The smooth NPMLE achieves the optimal nearly parametric rate for posterior mean estimation. Moreover, it achieves a polynomial convergence rate for prior and posterior density estimation under a compact support assumption on the mixing distribution. Based on the smooth NPMLE, we construct plug-in empirical Bayes confidence sets that mimic the oracle optimal (in terms of expected length) marginal coverage sets. We show theoretically and empirically that these sets achieve asymptotically exact marginal coverage and are substantially shorter than existing methods.
\end{abstract}


\spacingset{1} 
\section{Introduction}\label{sec:intro}
Consider the following Poisson mixture model:
\begin{equation}\label{eq:poisson model}
    X_i \mid \theta_i \overset{ind}\sim \mathrm{Poi}(\theta_i), \qquad \text{and} \qquad 
\theta_i \overset{iid}\sim G^*,\qquad \mbox{for }\;\; i = 1, \ldots, n,
\end{equation}
where the mixing distribution (or prior) $G^*$ is unknown and belongs to $\Pc(\Real_+)$, the set of all probability measures on $\Real_+ = [0,\infty)$. We observe the data $\{X_i\}_{i=1}^{n}$, and the $\{\theta_i\}_{i=1}^{n}$ are unobserved latent variables. This model has long been a canonical framework for modeling count data; see, e.g., \citep{Carriere1993, Patilea2024, Denuit2001SmoothedNE} for applications in actuarial sciences, and has been extensively studied in the literature \citep{Good1956, Simar1976, Lambert1984, Carriere1993, Patilea2024, Denuit2001SmoothedNE, Efron2010, efron2019, Ignatiadis2022}. Dating back to Herbert Robbins's pioneering work in the 1950s \citep{robbins1956empiricalbayes}, empirical Bayes (EB) methods have been widely studied to estimate or predict large collections of latent effects for \eqref{eq:poisson model} \citep{Simar1976, polyanskiy2021sharpregretboundsempirical, Banerjee2021, Jana2023, Jana2025, Shen2026, kang2026functionestimationempiricalbayes, han2025bestinggoodturingoptimalitynonparametric}. Despite the canonical role of \eqref{eq:poisson model} in statistics and machine learning, the optimality of EB rules in estimating $\{\theta_i \}_{i=1}^{n}$ has only recently been studied; see, e.g., \citet{polyanskiy2021sharpregretboundsempirical} for the Robbins estimator \citep{robbins1956empiricalbayes}, and \citet{Jana2025} for minimum-distance estimators including the nonparametric maximum likelihood estimator (NPMLE). Recent works \citep{teh2025solvingempiricalbayestransformers, cannella2026universalpriorssolvingempirical} study how a transformer pretrained on synthetic data achieves vanishing regret in the Poisson EB mean estimation problem.

However, comparatively little is known about \emph{uncertainty quantification}: how to construct optimal confidence sets for $\{ \theta_i \}_{i=1}^{n}$ under \eqref{eq:poisson model}? The classical confidence intervals of \citet{Garwood1936} are known to achieve frequentist coverage of $\{\theta_i\}_{i=1}^{n}$ under the Poisson mixture model \eqref{eq:poisson model}, but they do not use the prior information $G^*$ and are overly conservative \citep{Porter2026}. To make inference on $\{\theta_i\}_{i=1}^{n}$, one may estimate the prior $G^*$ and plug it into downstream statistical procedures; this approach is known as $g$-modeling \citep{Efron2014}. A classical choice for such a plug-in estimator of $G^*$ is the NPMLE, which maximizes the marginal likelihood of the observations $\{X_i \}_{i=1}^{n}$ over all probability distributions on $\Real_+$. While the NPMLE performs well for estimating $\{\theta_i\}_{i=1}^{n}$ without parametric assumptions on $G^*$ \citep{Jana2025,Shen2026}, it has a critical drawback for {\it inference}: it is always discrete \citep{Simar1976,Lindsay1995}. The discreteness of the NPMLE makes it difficult to recover the density of the true prior $G^*$; see the leftmost plot of Figure~\ref{fig:SNPML-Poisson}. Indeed, the convergence rate of the NPMLE under the Poisson mixture model \eqref{eq:poisson model} is logarithmically slow with respect to the $1$-Wasserstein distance \citep{Miao2024}. The induced plug-in posterior is also discrete, which complicates inference for $\{\theta_i \}_{i=1}^{n}$.

Our goal in this paper is to construct a smooth NPMLE of the prior density (as illustrated in the center plot of Figure~\ref{fig:SNPML-Poisson}). That gives strong theoretical guarantees for posterior mean estimation, prior density estimation and uncertainty quantification, yet remains computationally efficient. Based on the smooth NPMLE, we construct EB marginal coverage sets for $\{\theta_i\}_{i=1}^{n}$ that are optimal in expected length among all sets attaining a target coverage level.

\begin{figure}[t]
    \centering
    \begin{subfigure}[t]{0.32\textwidth}
        \centering
        \includegraphics[width=\linewidth]{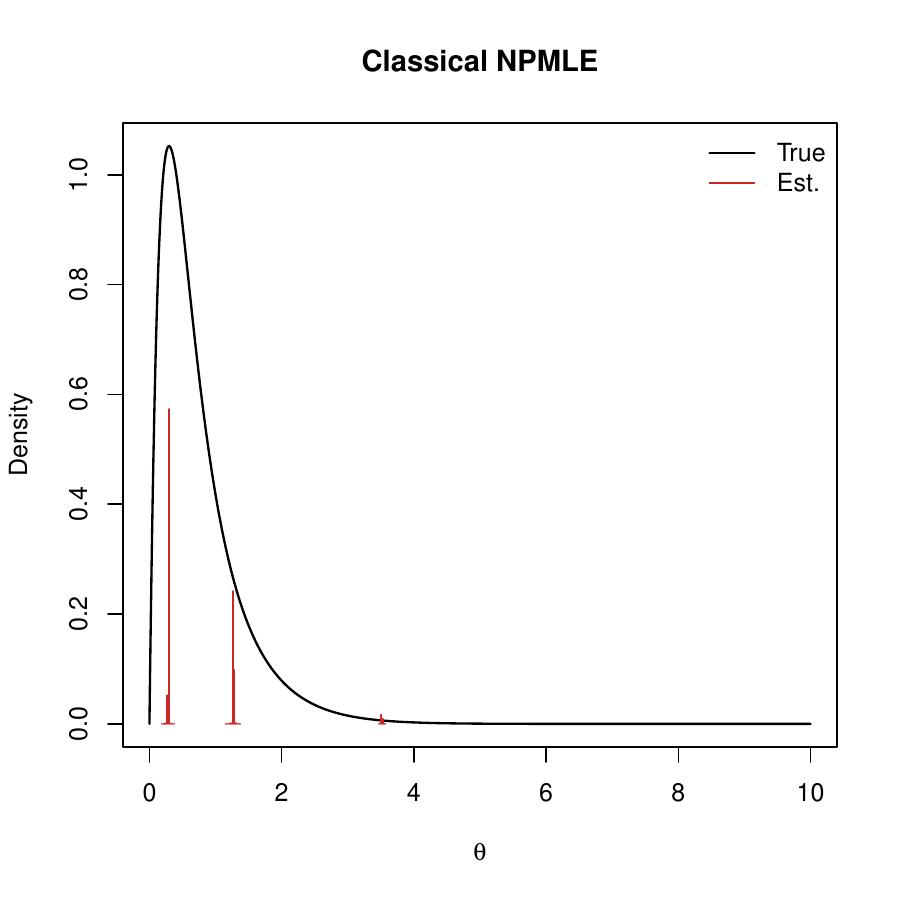}
    \end{subfigure}
    \hfill
    \begin{subfigure}[t]{0.32\textwidth}
        \centering
        \includegraphics[width=\linewidth]{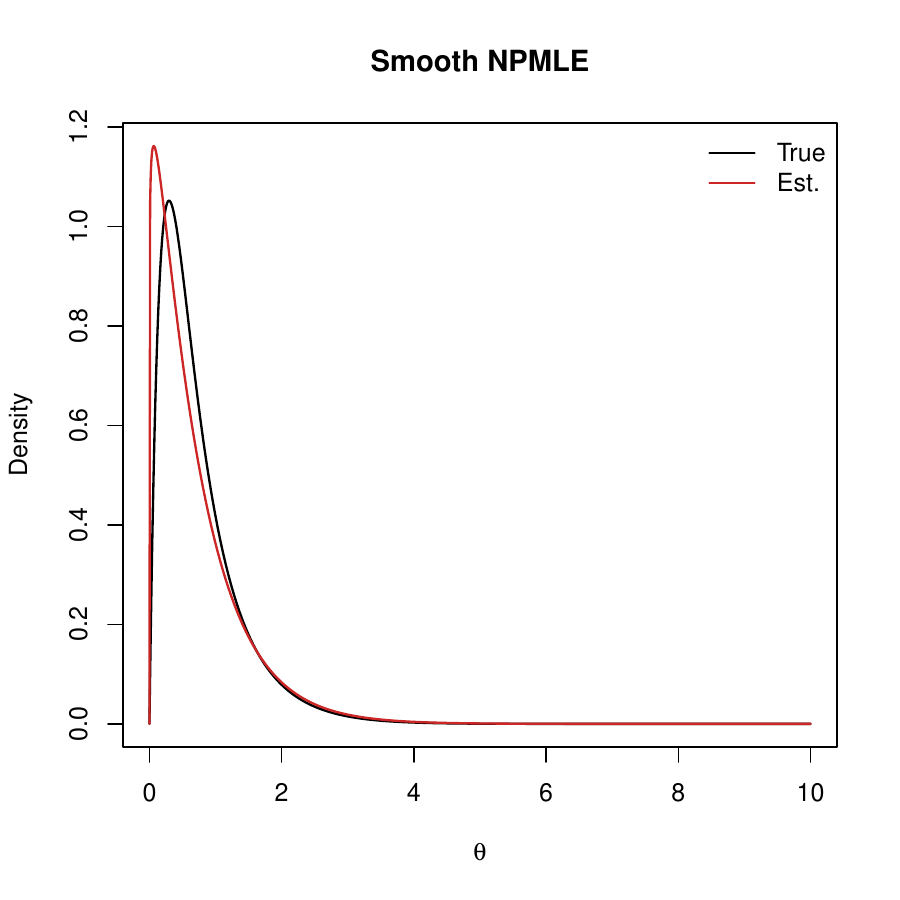}
    \end{subfigure}
    \hfill
    \begin{subfigure}[t]{0.32\textwidth}
        \centering
        \includegraphics[width=\linewidth]{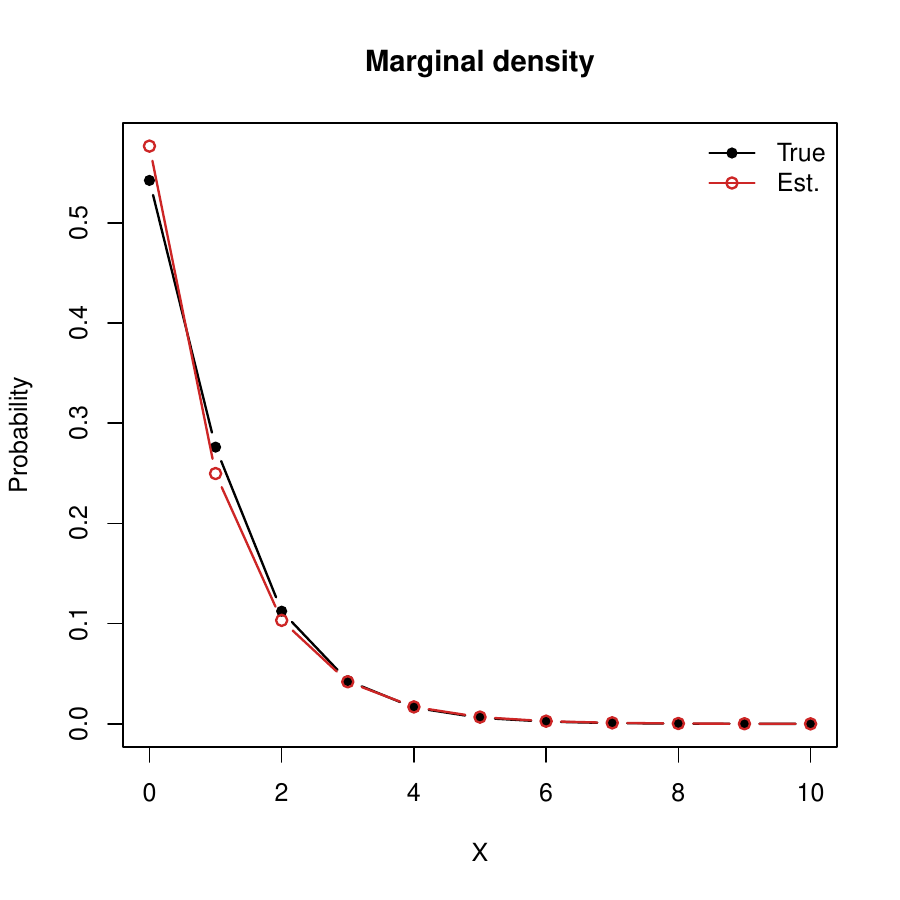}
    \end{subfigure}
    \caption{ We consider $G^* = \tfrac{1}{2}\,\mathrm{Gamma}(2,2) + \tfrac{1}{2}\,\mathrm{Gamma}(2,4)$ in \eqref{eq:poisson model} (equivalently, $\trueprior = \delta_{2}/2 + \delta_{4}/2$ and $\kappa_* = 2$ in \eqref{eq:poisson hierarchical model}). The classical (discrete) NPMLE and the smooth NPMLE \eqref{eq:poisson estimated prior density} are computed using $n = 1000$ observations and shown in the left and center plots along with the true prior density. The true marginal density of the observations and the estimated marginal density based on the smooth NPMLE are shown in the rightmost plot. For the smooth NPMLE, $\kappa_*$ in \eqref{eq:poisson hierarchical model} is also estimated using the neighborhood procedure described in Section~\ref{sec:identifiability-poisson}.}
    \label{fig:SNPML-Poisson}
\end{figure}

\subsection{Main contributions}
\begin{itemize}
    \item We introduce a smooth NPMLE for Poisson mixture models by adding a Gamma smoothing layer to the model (see \eqref{eq:poisson hierarchical model}). This choice is especially natural for the Poisson model: Gamma priors are conjugate to the Poisson likelihood, and the class of Gamma mixture priors is highly flexible and can approximate a broad range of continuous priors on $(0,\infty)$ \citep{Wiper2001, Bochkina2017}. The hierarchical model yields a smooth estimated prior density and a smooth posterior density estimate, while retaining the convex optimization structure of the classical NPMLE through the negative binomial mixture representation (Section~\ref{sec:poisson}).
    
    \item We establish the optimality of the smooth NPMLE for marginal density estimation and posterior mean estimation under mild conditions. We prove a nearly parametric bound for the marginal density estimation error (Theorem~\ref{thm:poisson marginal density estimation}) and, as a consequence, a nearly parametric regret bound for EB posterior mean estimation (Theorem~\ref{thm:poisson mean estimation}). Moreover, we show that the smooth NPMLE converges to the true prior density at a polynomial rate in total variation distance, up to logarithmic factors (Theorem~\ref{thm:poisson convergence of smooth NPMLE}). The plug-in posterior density converges at the same polynomial rate in weighted total variation distance (Theorem~\ref{thm:poisson posterior density convergence}).

    \item We show that oracle optimal marginal coverage sets in terms of expected length are obtained by thresholding the posterior density at a single global level for Poisson mixture models (Theorem~\ref{thm:poisson-opt-marginal}). We construct a smooth plug-in EB analogue based on the smooth NPMLE and prove that it attains the target marginal coverage level asymptotically (Theorem~\ref{thm:poisson-opt-coverage}). Our simulation results show that our procedure yields substantially shorter marginal coverage sets than existing methods while maintaining coverage close to the nominal level (Section~\ref{sec:simulation}).

    \item We study identifiability of the hierarchical Poisson model \eqref{eq:poisson hierarchical model} when the shape parameter in the Gamma mixture is unknown (Section~\ref{sec:identifiability-poisson}). In this case, the model is not identifiable (Lemma~\ref{lem:nestedness of Gamma mixture}). Nevertheless, we show that the smallest shape parameter for which the hierarchical representation \eqref{eq:poisson hierarchical model} holds is identifiable. We construct a consistent estimator of this parameter based on the neighborhood procedure of \citet{Donoho1988} (Theorem~\ref{thm:poisson-kappa-consistency}).

\end{itemize}

\subsection{Related work}

The closest work to ours is \citet{kim2026empiricalbayesestimationinference}, who develop the analogous smooth NPMLE framework for Gaussian location mixtures and study its theoretical properties. Their work and ours are motivated by the `bet on smoothness' principle, which restricts attention to a smooth subclass of nonparametric priors; see, e.g., \citep{efron2016, Stephens2016, Bovy2011, Cordy1997, Shen1999}. Our contribution is to show that the same smooth NPMLE principle can be made to work in the discrete Poisson model where the likelihood, conjugacy, approximation theory, and identifiability structure are fundamentally different. In particular, the Gamma-Poisson conjugacy leads to a negative binomial mixture likelihood. For the posterior mean estimation problem, our arguments build on a line of work \citep{polyanskiy2021sharpregretboundsempirical, Jana2025} for Poisson mixture models, while \citet{kim2026empiricalbayesestimationinference} build on \citep{Jiang2009, Saha2020, Soloff2025} for Gaussian location mixture models. For the prior estimation problem, our proof techniques are inspired by \citep{Miao2024, Han2023}, who show that the NPMLE achieves a polynomial convergence rate under the Gaussian-smoothed optimal transport distance. In our setting, the key approximation step relies on Laguerre polynomial approximation \citep{Szego1975}. In contrast, \citet{kim2026empiricalbayesestimationinference} leverage Fourier-analytic techniques with the fast convergence rate for marginal density estimation.

While estimation of latent effects has been extensively studied in EB \citep{Brown2009, polyanskiy2021sharpregretboundsempirical, Jana2023, Jana2025, kang2026functionestimationempiricalbayes, han2025bestinggoodturingoptimalitynonparametric, Jiang2009, Saha2020, Soloff2025}, uncertainty quantification remains relatively underexplored beyond strongly parametric settings. Early works \citep{Morris1983, Laird1987} propose EB confidence intervals under a Gaussian prior assumption for Gaussian mixture models, but their validity relies on strong parametric assumptions. A recent work by \citet{Armstrong2022} proposes robust EB confidence intervals (REBCIs) that maintain nominal coverage over all prior distributions satisfying certain moment conditions. For the Poisson mixture model \eqref{eq:poisson model}, they calibrate a family of intervals that bridges the equal-tailed posterior credible interval under a conjugate Gamma prior and the Garwood interval \citep{Garwood1936} (see Appendix G.2 therein). The resulting interval is substantially shorter than the Garwood interval. However, their procedure can still be conservative and yield intervals that are longer than necessary. In contrast, our procedure uses Gamma mixture priors to construct EB marginal coverage sets that achieve substantially shorter length while still maintaining the target coverage level.

\section{Hierarchical Poisson model and smooth NPMLE}\label{sec:poisson}
In this section, we introduce a hierarchical Poisson mixture model and a smooth NPMLE to overcome the limitations of the classical NPMLE. Specifically, we assume that the true prior $G^*$ is itself a Gamma mixture and add a Gamma smoothing layer to \eqref{eq:poisson model}: for $i = 1, \ldots, n$,
\begin{equation}\label{eq:poisson hierarchical model}
    X_i \mid \theta_i \overset{ind}\sim \mathrm{Poi}(\theta_i), \quad \qquad 
\theta_i \mid \lambda_i \overset{ind}\sim \mathrm{Gamma}(\kappa_*, \lambda_i), \quad \qquad \lambda_i \overset{iid}\sim \trueprior,
\end{equation}
where, for now, $\kappa_* > 0$ is assumed to be known and $\trueprior \in \Pc((0,\infty))$ is unknown. Throughout, we use the shape-rate parametrization for the Gamma distribution. Note that the Gamma distribution is a conjugate prior for the Poisson likelihood. Moreover, Gamma mixtures have been widely used  to model mixing distributions for count data (see, e.g., \citep{Maritz1966, Denuit2001SmoothedNE}), and are sufficiently flexible to approximate a broad class of continuous priors on $(0,\infty)$ \citep{Wiper2001, Bochkina2017}. As we discuss later in Section~\ref{sec:identifiability-poisson}, the Gamma mixture prior family induced by \eqref{eq:poisson hierarchical model} is nested in the shape parameter. Consequently, fixing $\kappa_*$ in \eqref{eq:poisson hierarchical model} is still flexible, since the model can represent any prior that admits a representation with shape parameter at most $\kappa_*$. We relax the assumption that $\kappa_*$ is known in that section.

Even though we introduce the additional layer in the hierarchical model  \eqref{eq:poisson hierarchical model}, it still preserves the computational advantages of the classical NPMLE under \eqref{eq:poisson model}. Under \eqref{eq:poisson hierarchical model}, we have
\begin{align}\label{eq:negative binomial model}
    X_i \mid \lambda_i \overset{ind}\sim \mathrm{NB}\left(\kappa_*, \frac{\lambda_i}{\lambda_i+1}\right), \quad \qquad \lambda_i \overset{iid}\sim \trueprior
\end{align}
where $\mathrm{NB}(r,p)$ denotes the negative binomial distribution.
We estimate $\trueprior$ via the NPMLE approach, defined as any maximizer
\begin{equation}\label{eq:poisson NPMLE}
    \npmle \in \argmax_{H\in\Pc((0,\infty])} \sum_{i=1}^{n} \log f_{H}(X_i),
\end{equation}
where $f_{H}$ is the marginal density of $X_i$ induced by $H$ under \eqref{eq:poisson hierarchical model}:
\begin{align}\label{eq:poisson marginal density}
    f_{H}(x) \;:=\; \int r_{\kappa_*,\lambda}(x)\, \diff H(\lambda), \qquad r_{\kappa_*,\lambda}(x) :=\frac{\Gamma(x+\kappa_*)}{x! \Gamma(\kappa_*)} \left(\frac{1}{\lambda+1} \right)^x\left(\frac{\lambda}{\lambda+1} \right)^{\kappa_*},
\end{align}
for $x = 0, 1, 2, \ldots$. Here, for $\lambda=\infty$, we define $r_{\kappa_*,\lambda}(x)=\1v(x=0)$. When $\kappa_*$ is known, the negative binomial distribution belongs to a one-dimensional exponential family. Thus, the structural properties of the NPMLE are well understood \citep{Simar1976, Lindsay1995, Balabdaoui2025}: the NPMLE $\npmle$ exists and is unique, has at most $n$ atoms, and is supported on $[\kappa_*/X_{(n)},\infty]$ where $X_{(n)} = \max_i X_i$. Even though this is an infinite-dimensional convex optimization problem, it can be well approximated by a finite-dimensional convex program that is essentially tuning-free \citep{Koenker2014,Koenker2017}. 

A main advantage of using model \eqref{eq:poisson hierarchical model} is that it yields a smooth estimate of the true prior density. Note that, under \eqref{eq:poisson hierarchical model}, the marginal density of $\theta_i$ takes the form
\begin{align}\label{eq:poisson true prior density}
    g_{\trueprior}(\theta) \;:=\; \int_{(0,\infty)} \frac{\lambda^{\kappa_*}}{\Gamma(\kappa_*)}\theta^{\kappa_*-1}e^{-\lambda\theta}\, \diff \trueprior(\lambda), \qquad \theta > 0,
\end{align}
which is a mixture of Gamma distributions. The NPMLE $\npmle$ in \eqref{eq:poisson NPMLE} yields a natural plug-in estimator of the smooth prior density in \eqref{eq:poisson true prior density}:
\begin{equation}\label{eq:poisson estimated prior density}
    g_{\npmle}(\theta) \;:=\; \int_{(0,\infty)} \frac{\lambda^{\kappa_*}}{\Gamma(\kappa_*)}\theta^{\kappa_*-1}e^{-\lambda\theta}\, \diff \npmle(\lambda), \qquad \theta > 0
\end{equation}
which we call the {\it smooth NPMLE}. See the center plot of Figure~\ref{fig:SNPML-Poisson} where the true prior density is approximated quite well by the smooth NPMLE. Hence, by using the NPMLE under the negative binomial model \eqref{eq:negative binomial model}, we obtain a smooth estimate of the prior density for the Poisson model.

\begin{remark}[Interpretation of the smooth NPMLE]
Note that the NPMLE $\npmle$ in \eqref{eq:poisson NPMLE} is computed over $(0,\infty]$. For example, when many observations are zero, $\npmle$ may place positive mass at $\infty$, so the induced prior measure has an atom at $0$. Here, $g_{\npmle}$ denotes the density of the absolutely continuous part of the induced prior measure on $(0,\infty)$. Nevertheless, since the true prior is assumed to be a Gamma mixture density, it is natural to define the smooth NPMLE through the smooth component of the induced prior measure. In practice, one may also compute the NPMLE over $[\kappa_*/X_{(n)},\,U]$ for a sufficiently large $U>0$, in which case $g_{\npmle}$ is a proper smooth estimated prior density.
\end{remark}

\section{Statistical properties of the smooth NPMLE}\label{sec:poisson details}
In this section, we study the theoretical properties of the smooth NPMLE under the hierarchical Poisson model \eqref{eq:poisson hierarchical model}. Throughout, we write $x \lesssim_{p,q} y$ to mean that $x \le C_{p,q}\,y$ for a constant $C_{p,q}>0$ depending only on the parameters $p$ and $q$. We write $\Pc(A)$ for the set of all probability measures on $A \subseteq \Real$. Let $\E_{H}$ and $\P_{H}$ denote expectation and probability under \eqref{eq:poisson hierarchical model} with $\lambda_i \overset{iid}\sim H \in \Pc((0,\infty))$.

\subsection{Posterior mean estimation}\label{sec:mean estimation}
In this subsection, we demonstrate that our smooth NPMLE achieves optimal nearly parametric regret rates for the Poisson mixture model under mild conditions. It is well-known that, for the Poisson mixture model, the posterior mean can be expressed in terms of the marginal density $f_{H}$ in \eqref{eq:poisson marginal density}. That is, under \eqref{eq:poisson hierarchical model} with $\lambda_i \overset{iid}{\sim} H$, the posterior mean of $\theta_i$ given $X_i$ is \citep{robbins1956empiricalbayes}:
\begin{align}\label{eq:poisson posterior mean}
    \hat\theta_H(x):= \E_{H}[\theta_i \mid X_i = x]  = (x+1) \frac{f_{H}(x+1)}{f_{H}(x)}
\end{align}
whenever $f_H(x)>0$. The oracle posterior mean of $\theta_i$ is $\hat\theta_{\trueprior}(X_i)$ obtained by replacing $f_H$ in \eqref{eq:poisson posterior mean} with $f_{\trueprior}$. Its EB counterpart is $\hat\theta_{\npmle}(X_i)$ obtained by replacing $f_H$ with $f_{\npmle}$. Before we establish the nearly parametric regret rates, we prove a nearly parametric bound for marginal density estimation in Hellinger distance $\Hellinger^2(p,q) := \sum_{x=0}^\infty (\sqrt{p(x)}-\sqrt{q(x)})^2/2$. We provide its proof in Appendix~\ref{app:proof of thm:poisson marginal density estimation}.

\begin{theorem}\label{thm:poisson marginal density estimation}
    Suppose that \eqref{eq:poisson hierarchical model} holds where $\kappa_* > 0$ and $\trueprior \in \Pc([L, \infty))$ for some $L > 0$. Let $\npmle$ be any solution of \eqref{eq:poisson NPMLE}. Then,
    \begin{align}\label{eq:poisson density estimation}
        \E_{\trueprior}[\Hellinger^2(f_{\npmle}, f_{\trueprior})] \lesssim_{\kappa_*, L} \frac{\log n}{n}.
    \end{align}
\end{theorem}

As in the rightmost plot of Figure~\ref{fig:SNPML-Poisson}, the estimated marginal density approximates the true marginal density well. The next theorem leverages Theorem~\ref{thm:poisson marginal density estimation} to obtain a nearly parametric regret bound, up to logarithmic factors, for posterior mean estimation. Throughout, for probability mass functions $p$ and $q$, we write
$\mathrm{TV}(p,q) := \sum_{x=0}^{\infty} |p(x)-q(x)|/2$. We provide its proof in Appendix~\ref{app:proof of thm:poisson mean estimation}.

\begin{theorem}\label{thm:poisson mean estimation}
    Under the assumptions of Theorem~\ref{thm:poisson marginal density estimation}, let $X \sim f_{\trueprior}$ be independent of $\npmle$. Then,
    \begin{align}\label{eq:poisson regret bound}
        \E_{\trueprior}[(\hat\theta_{\npmle}(X) - \hat\theta_{\trueprior}(X))^2] \lesssim_{\kappa_*, L} \frac{ (\log n)^{3}}{n}.
    \end{align}
\end{theorem}
\begin{remark}[Interpretation of the regret]
    In Theorem~\ref{thm:poisson mean estimation}, we use an EB regret for an observation $X$ independent of the training sample used to construct $\npmle$. This is a standard formulation in the Poisson EB literature \citep{polyanskiy2021sharpregretboundsempirical, Jana2023, Jana2025}. By exchangeability, this is the average regret of the EB rule that estimates $\theta_i$ using the leave-one-out NPMLE computed from the sample without the $i$th observation.
\end{remark}

Note that Theorems~\ref{thm:poisson marginal density estimation} and \ref{thm:poisson mean estimation} are proved for the case where $\npmle$ in \eqref{eq:poisson NPMLE} is computed over the unconstrained class $\Pc((0,\infty])$. Also,  when $\trueprior \in \Pc([L, \infty))$ for some $L > 0$ in \eqref{eq:poisson hierarchical model}, the true prior $G^*$ belongs to the class of subexponential priors $ \mathrm{SubE}(s) := \{G : G([t, \infty)) \le 2e^{-t/s}, \;\forall t \ge 0 \}$  for some $s = s(\kappa_*, L)> 0$. For this case, Theorems 21 and 2 of \citet{polyanskiy2021sharpregretboundsempirical} show that the rates in \eqref{eq:poisson density estimation} and \eqref{eq:poisson regret bound} are minimax optimal with exact logarithmic factors.

\begin{remark}[Proof strategy and model difference]
    The proofs of Theorems~\ref{thm:poisson marginal density estimation} and \ref{thm:poisson mean estimation} adapt the arguments of Theorems 2 and 3 in \citet{Jana2025}, but the setting is fundamentally different. While \citet{Jana2025} study minimum-distance estimators for the Poisson mixture model \eqref{eq:poisson model}, we work with the hierarchical Poisson model \eqref{eq:poisson hierarchical model} and use the NPMLE under the negative binomial mixture model \eqref{eq:negative binomial model} to construct smooth estimates of the prior and posterior densities.
\end{remark}

\subsection{Prior and posterior density estimation}\label{sec:prior and posterior estimation}
In this subsection, we study the convergence of the smooth NPMLE and the corresponding posterior density. Under the Poisson mixture model \eqref{eq:poisson model}, it has long been known that estimation of the mixing distribution is quite hard; the minimax lower bound under metrics such as the total variation distance or the $1$-Wasserstein distance is logarithmic \citep{Zhang1995, Loh1996, Han2023, Miao2024}. In particular, \citet{Miao2024} show that the NPMLE achieves only a logarithmic convergence rate under the $1$-Wasserstein distance. Our result shows that, when we focus on estimating Gamma mixture prior densities under boundedness assumptions on the rate parameter, the smooth NPMLE converges to the true prior density in total variation distance at a polynomial rate. We provide the proof of Theorem~\ref{thm:poisson convergence of smooth NPMLE} in Appendix~\ref{app:proof of thm:poisson convergence of smooth NPMLE}.

\begin{theorem}\label{thm:poisson convergence of smooth NPMLE}
Suppose that \eqref{eq:poisson hierarchical model} holds where $\kappa_* > 0$ and $\trueprior \in \Pc([L,U])$ for some $0 < L < U < \infty$. Let $\npmle$ be any solution of \eqref{eq:poisson NPMLE} where the maximization is taken over $\Pc([L,U])$. Then there exists $\alpha_* \in (0,1/2)$, depending only on $(L,U)$, such that
\begin{align*}
    \E_{\trueprior}\left[ \mathrm{TV}(g_{\npmle},g_{\trueprior})\right] \lesssim_{\kappa_*, L, U} n^{-\alpha_*} (\log n)^{(1-\kappa_*)_+ + \alpha_*}.
\end{align*}
\end{theorem}

\begin{remark}[Compact support assumption]\label{rem:compact support assumption}
   In Theorem~\ref{thm:poisson convergence of smooth NPMLE}, the NPMLE $\npmle$ is computed over the same known interval $[L,U]$ on which the true prior $\trueprior$ is supported. This technical assumption is similar in spirit to those in \citep{Miao2024,Han2023} under \eqref{eq:poisson model}, where the true mixing distribution is assumed to be supported on a known compact interval and the NPMLE is computed over the same class. In our context, the lower bound $L$ rules out heavy-tailed Gamma mixture priors on $\theta$ as in Theorems~\ref{thm:poisson marginal density estimation} and~\ref{thm:poisson mean estimation}, while the upper bound $U$ ensures that Gamma mixture priors are not arbitrarily concentrated near $0$. See Remark~\ref{rem:alphastar} in Appendix~\ref{app:proof of thm:poisson convergence of smooth NPMLE} for the explicit dependence of $\alpha_*$ on $(L,U)$.
\end{remark}

In Appendix~\ref{app:illustration-smooth-NPMLE}, we illustrate the polynomial convergence rate of the smooth NPMLE via simulations. We also compare this rate with the polynomial convergence rates obtained under the Gaussian-smoothed optimal transport framework \citep{Han2023,teh2023} and discuss the differences in our proof strategy.

Our convergence result for the smooth NPMLE also implies convergence of the posterior density. Under \eqref{eq:poisson hierarchical model}, the posterior density of $\theta_i$ given $X_i=x$ is
\begin{align}\label{eq:poisson posterior density}
    \pi_{\trueprior}(\theta \mid x) := \frac{p_{\theta}(x)g_{\trueprior}(\theta)}{f_{\trueprior}(x)}, \qquad \theta > 0, \qquad x = 0,1,2,\cdots,
\end{align}
where $p_{\theta}(x) := e^{-\theta}\theta^x/x!$, $g_{\trueprior}$ is defined in \eqref{eq:poisson true prior density}, and $f_{\trueprior}$ is defined via \eqref{eq:poisson marginal density}. Replacing the true prior density $g_{\trueprior}$ by the smooth NPMLE $g_{\npmle}$ yields the estimated posterior density
\begin{align}\label{eq:poisson estimated posterior density}
    \pi_{\npmle}(\theta \mid x) := \frac{p_{\theta}(x)g_{\npmle}(\theta)}{f_{\npmle}(x)}, \qquad \theta > 0,\qquad x = 0,1,2,\cdots,
\end{align}
where $f_{\npmle}$ is the estimated marginal density of $X_i$. We measure the distance between the true and estimated posterior densities using the weighted total variation distance $\mathrm{wTV}(\pi_{\npmle},\pi_{\trueprior}) := \sum_{x=0}^{\infty} \mathrm{TV}\big(\pi_{\npmle}(\cdot \mid x), \pi_{\trueprior}(\cdot \mid x)\big)f_{\trueprior}(x)$. We provide the proof of Theorem~\ref{thm:poisson posterior density convergence} in Appendix~\ref{app:proof of thm:poisson posterior density convergence}.
\begin{theorem}\label{thm:poisson posterior density convergence}
    Under the assumptions of Theorem~\ref{thm:poisson convergence of smooth NPMLE} with the same $\alpha_* \in (0,1/2)$,
    \begin{align*}
        \E_{\trueprior}\left[\mathrm{wTV}(\pi_{\npmle},\pi_{\trueprior})\right] \lesssim_{\kappa_*, L, U} n^{-\alpha_*} (\log n)^{(1-\kappa_*)_+ + \alpha_*}.
    \end{align*}
\end{theorem}

\section{Optimal marginal coverage sets}\label{sec:OPT}
\subsection{Oracle optimal marginal coverage sets}
In this section, we discuss how to construct confidence sets for $\{\theta_i \}_{i=1}^{n}$ based on the smooth NPMLE. We focus on the notion of {\it marginal coverage} \citep{Morris1983, Armstrong2022, Hoff2025, kim2026empiricalbayesestimationinference}, which is quite natural in EB problems because the latent effects $\{\theta_i \}_{i=1}^{n}$ are assumed to follow a common prior distribution $G^*$. We say a set-valued rule $\Ic:x \mapsto \Ic(x) \subset (0,\infty)$ maintains $(1-\beta)$ marginal coverage if
\begin{align}\label{eq:poisson marginal coverage}
    \P_{\trueprior}(\theta_i \in \Ic(X_i)) \ge 1-\beta,
\end{align}
where the probability is computed under the joint law of $(\theta_i, X_i)$ in \eqref{eq:poisson hierarchical model}. Unlike conditional coverage, which is required for each fixed parameter value $\theta_i$ in the frequentist sense or each fixed observation value $X_i = x$ in the Bayesian sense, marginal coverage only requires the joint coverage probability to be at least $1-\beta$.  Among all measurable set-valued rules maintaining $(1-\beta)$ marginal coverage, we seek one with the smallest expected length:
\begin{align}\label{eq:poisson optimal marginal set problem}
    \min_{\Ic(\cdot)} \E_{\trueprior}[|\Ic(X_i)|] \qquad \mbox{subject to} \qquad \P_{\trueprior}(\theta_i \in \Ic(X_i)) \ge 1-\beta,
\end{align}
where $|A|$ denotes the Lebesgue length of a measurable set $A \subset (0,\infty)$. Theorem~\ref{thm:poisson-opt-marginal} shows that the oracle solution is obtained by thresholding the true posterior density at a global level. We provide its proof in Appendix~\ref{app:proof of thm:poisson-opt-marginal}, where we verify the assumptions of Theorem 3.1 of \citet{kim2026empiricalbayesestimationinference} for our Poisson hierarchical model \eqref{eq:poisson hierarchical model}.

\begin{theorem}\label{thm:poisson-opt-marginal}
    Suppose that \eqref{eq:poisson hierarchical model} holds where $\kappa_* > 0$ and $\trueprior \in \Pc((0,\infty))$. Let $0 < \beta < 1$. Then there exists a unique constant $k^* \ge 0$ such that the set-valued rule
    \begin{align}\label{eq:poisson oracle optimal set}
        \Ic^*(x) := \{\theta \ge 0 : \pi_{\trueprior}(\theta \mid x) \ge k^*\}, \qquad x = 0,1,2,\cdots,
    \end{align}
    with $\pi_{\trueprior}(\theta \mid x)$ defined in \eqref{eq:poisson posterior density} satisfies
    \begin{align}\label{eq:poisson oracle threshold exact}
        \P_{\trueprior}(\theta_i \in \Ic^*(X_i)) = 1-\beta.
    \end{align}
    Moreover, $\Ic^*$ solves \eqref{eq:poisson optimal marginal set problem} up to Lebesgue-null sets.
\end{theorem}

Because $\Ic^*$ is obtained by thresholding the posterior density, it induces a natural nested family of optimal sets as the target coverage level $1-\beta$ varies. Moreover, in contrast to highest posterior density (HPD) sets which use an $x$-dependent threshold to enforce conditional posterior content, the optimal marginal coverage sets in \eqref{eq:poisson oracle optimal set} are defined by a single global threshold $k^*$. This makes it easy to construct the coverage sets simultaneously for all observation values.

\subsection{Estimation of optimal marginal coverage sets}
We construct a plug-in estimate of the oracle optimal marginal coverage set $\Ic^*$ in \eqref{eq:poisson oracle optimal set} using the smooth NPMLE. First, we approximate the true posterior density $\pi_{\trueprior}(\theta \mid x)$ in \eqref{eq:poisson posterior density} by the estimated posterior density $\pi_{\npmle}(\theta \mid x)$ in \eqref{eq:poisson estimated posterior density}. Theorem~\ref{thm:poisson posterior density convergence} shows that this approximation is accurate.
 Next, motivated by \eqref{eq:poisson oracle threshold exact}, we estimate the oracle threshold $k^*$ by the plug-in quantity
\begin{align}\label{eq:poisson estimated threshold}
    \hat{k}_n := \sup\Bigg\{k \ge 0 :
    \sum_{x=0}^{\infty}\int_0^{\infty}
    \1v\big(\pi_{\npmle}(\theta \mid x) \ge k\big)
    p_{\theta}(x) g_{\npmle}(\theta)\diff \theta
    \ge 1-\beta \Bigg\},
\end{align}
Equivalently, $\hat{k}_n$ is obtained by replacing $\trueprior$ in \eqref{eq:poisson oracle optimal set} and \eqref{eq:poisson oracle threshold exact} with $\npmle$. We then define the EB marginal coverage set
\begin{align}\label{eq:poisson estimated optimal set}
    \hat{\Ic}_n(x) := \{\theta \ge 0 : \pi_{\npmle}(\theta \mid x) \ge \hat{k}_n\}.
\end{align}
The following theorem shows that the EB marginal coverage set \eqref{eq:poisson estimated optimal set} attains the target marginal coverage level $1-\beta$ asymptotically. We provide its proof in Appendix~\ref{app:proof of thm:poisson-opt-coverage}.

\begin{theorem}\label{thm:poisson-opt-coverage}
    Under the assumptions of Theorem~\ref{thm:poisson convergence of smooth NPMLE} with the same $\alpha_* \in (0,1/2)$, let $(X,\theta)$ be independent of $\{(X_i,\theta_i)\}_{i=1}^{n}$ and distributed according to the joint law induced by \eqref{eq:poisson hierarchical model}. Then,
    \begin{align}\label{eq:poisson-opt-coverage}
        \E_{\trueprior}\left[\left|\P_{\trueprior}(\theta \in \hat{\Ic}_n(X)\mid \npmle) - (1-\beta)\right|\right]
        \lesssim_{\kappa_*,L,U}
        n^{-\alpha_*}(\log n)^{(1-\kappa_*)_+ + \alpha_*}.
    \end{align}
    Here, $\P_{\trueprior}(\theta \in \hat{\Ic}_n(X)\mid \npmle)$ denotes the conditional marginal coverage probability of $\hat{\Ic}_n$ given $\npmle$, i.e., $\P_{\trueprior}(\theta \in \hat{\Ic}_n(X)\mid \npmle) = \sum_{x=0}^{\infty}\int_{0}^{\infty} \1v\bigl(\theta \in \hat{\Ic}_n(x)\bigr)\, p_{\theta}(x)g_{\trueprior}(\theta)\,\diff\theta$.
\end{theorem}

A key advantage of the smooth NPMLE is that it can yield non-trivial confidence sets, since the estimated posterior density $\pi_{\npmle}(\theta \mid x)$ is smooth. By contrast, under the classical NPMLE for \eqref{eq:poisson model}, the posterior distribution is discrete, so the corresponding plug-in sets may be harder to interpret.

\begin{remark}[Computation of the threshold]
Instead of evaluating the sum and integral in \eqref{eq:poisson estimated threshold}, we may estimate $\hat{k}_n$ via Monte Carlo simulation: draw $\tilde{\lambda}_b \overset{iid}\sim \npmle$, then $\tilde{\theta}_b \mid \tilde{\lambda}_b \sim \mathrm{Gamma}(\kappa_*, \tilde{\lambda}_b)$ and $\tilde{X}_b \mid \tilde{\theta}_b \sim \mathrm{Poi}(\tilde{\theta}_b)$. We can then estimate $\hat{k}_n$ by the lower $\beta$-quantile of $\{\pi_{\npmle}(\tilde{\theta}_b \mid \tilde{X}_b)\}_{b=1}^{B}$.
\end{remark}
\begin{figure}[t]
    \centering
    \begin{subfigure}[t]{0.49\linewidth}
        \centering
        \includegraphics[width=\linewidth]{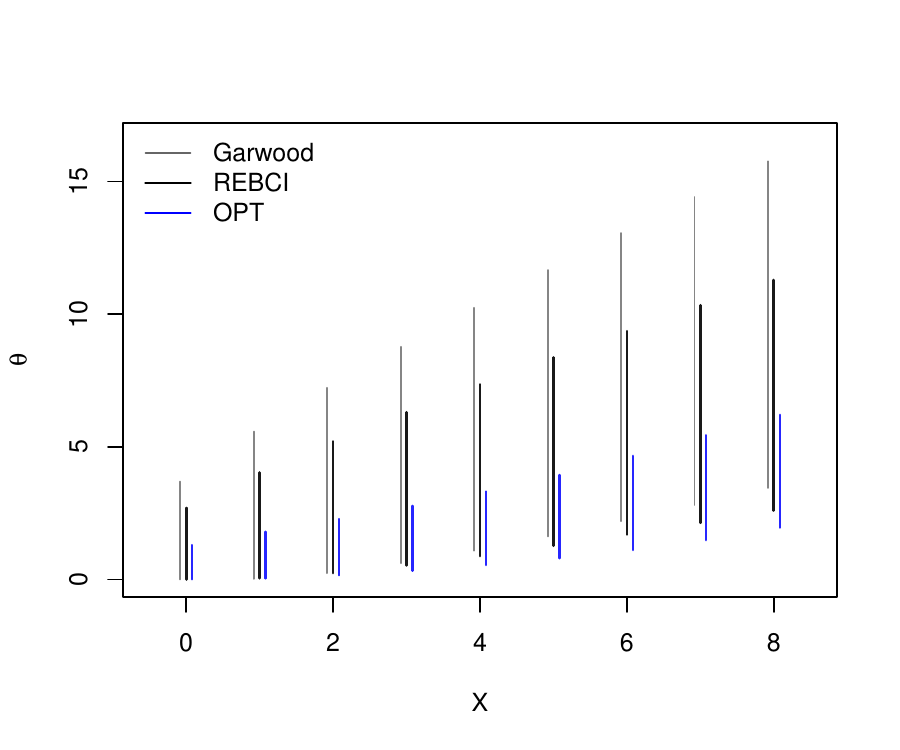}
        \caption{$G^* = \mathrm{Gamma}(2,2)/2 + \mathrm{Gamma}(2,4)/2$}
    \end{subfigure}
    \hfill
    \begin{subfigure}[t]{0.49\linewidth}
        \centering
        \includegraphics[width=\linewidth]{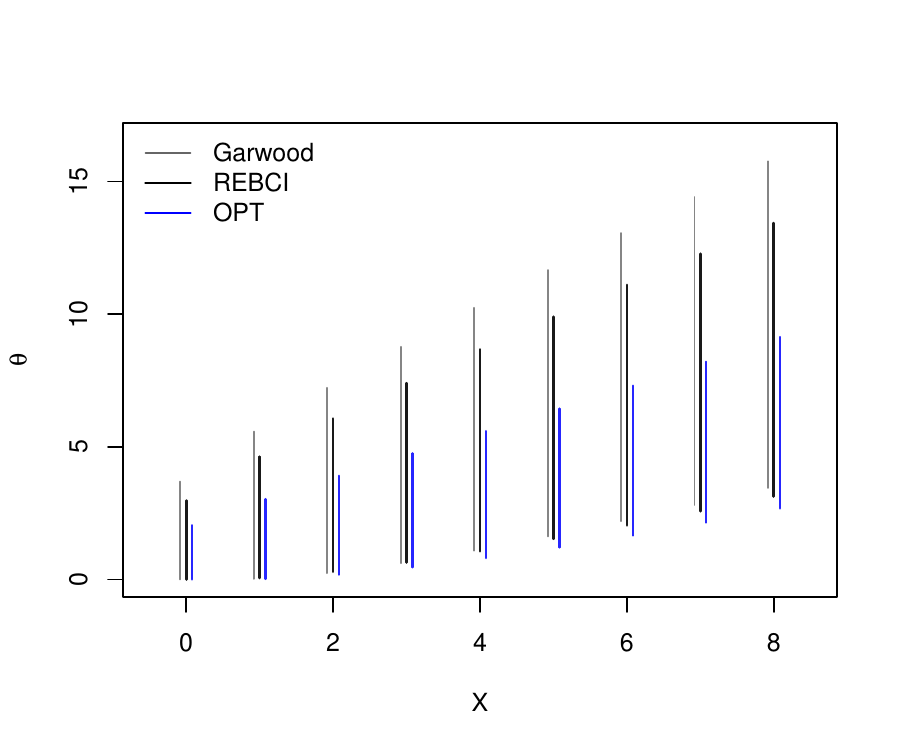}
        \caption{$G^* = \mathrm{Lognormal}(0,1)$}
    \end{subfigure}
    \caption{Comparison of our EB marginal coverage sets (OPT) with the Garwood confidence intervals \citep{Garwood1936} and the robust EB confidence intervals (REBCI) of \citet{Armstrong2022} at the nominal level $0.95$. The left panel uses the same prior as Figure~\ref{fig:SNPML-Poisson} and the right panel uses a misspecified lognormal prior.}
    \label{fig:OPT Comparison}
\end{figure}

Figure~\ref{fig:OPT Comparison} illustrates the substantial gain achieved by the EB marginal coverage sets. Our procedure is noticeably shorter than the competing methods, even under misspecification; $\mathrm{Lognormal}(0,1)$ does not belong to the Gamma mixture prior family. In Section~\ref{sec:simulation}, we demonstrate that our procedure yields the shortest sets while maintaining the desired coverage via simulation.

\section{Identifiability of the hierarchical Poisson model}\label{sec:identifiability-poisson}
So far, we have assumed that $\kappa_*$ in the hierarchical Poisson model \eqref{eq:poisson hierarchical model} is known. In this section, we discuss the case where $\kappa_*$ is unknown. In this case, the hierarchical Poisson model \eqref{eq:poisson hierarchical model} is non-identifiable. This is because the Gamma mixture prior family induced by \eqref{eq:poisson hierarchical model} is nested in the shape parameter as shown in the following lemma. We prove Lemma~\ref{lem:nestedness of Gamma mixture} in Appendix~\ref{app:proof of lem:nestedness of Gamma mixture}.

\begin{lemma}\label{lem:nestedness of Gamma mixture}
For each $\kappa > 0$, let
\begin{align}\label{eq:gamma mixture nested model}
    \Gc_{\kappa} := \{g_{H, \kappa}: H \in \Pc((0,\infty) ) \}, \qquad g_{H,\kappa}(\theta) := \int \frac{\lambda^{\kappa}}{\Gamma(\kappa)} \theta^{\kappa - 1} e^{-\lambda \theta} \diff H(\lambda), \quad \theta > 0.
\end{align}
Then for any $0<\kappa<\kappa'$, we have $\Gc_\kappa \subseteq \Gc_{\kappa'}$.
\end{lemma}

Consequently, there may be multiple choices of $(\kappa, H)$ that induce the same prior density $g_{\trueprior}$ in \eqref{eq:poisson true prior density}. Therefore, we target the quantity
\begin{align}\label{eq:kappa0}
    \kappa_0 \equiv \kappa_0(G^*) := \inf\{\kappa > 0: g_{\trueprior} \in \Gc_{\kappa} \}
\end{align}
where $g_{\trueprior}$ is the true prior density \eqref{eq:poisson true prior density} and $\Gc_{\kappa}$ is defined in \eqref{eq:gamma mixture nested model}. That is, $\kappa_0$ is trying to capture the least amount of Gamma smoothing needed to represent the true prior. We note that, under \eqref{eq:poisson hierarchical model}, $\kappa_0>0$ and the infimum in \eqref{eq:kappa0} is indeed attained; see Appendix~\ref{app:proof of smallest shape parameter}.

\subsection{Estimation of $\kappa_0$ via the neighborhood procedure}
A natural question is how to estimate $\kappa_0$. Our proposed procedure builds on the neighborhood procedure introduced by \citet{Donoho1988}; see also \citep{Donoho2013, kim2026empiricalbayesestimationinference}. To this end, we express $\kappa_0$ in \eqref{eq:kappa0}, which is defined through the prior, as a functional of the observable distribution. For $H\in\Pc((0,\infty))$ and $\kappa > 0$, let $f_{H,\kappa}$ be the marginal probability density function obtained by replacing $\kappa_*$ with $\kappa$ in \eqref{eq:poisson marginal density}. For each $\kappa > 0$, define
\begin{align}\label{eq:observable distributions}
    \Fc_{\kappa} := \{F_{H,\kappa}: H \in \Pc((0,\infty) )\}, \qquad F_{H,\kappa}(m) := \sum_{x=0}^m f_{H,\kappa}(x), \quad m \in \mathbb{Z}_+.
\end{align}
That is, $\Fc_\kappa$ is the class of observable count distribution functions induced by Gamma mixture priors with shape parameter $\kappa$. Let $F^*$ denote the true cumulative distribution function of $X_1$ under \eqref{eq:poisson hierarchical model}. Then, by identifiability of Poisson mixtures \citep{Teicher1961}, the observable representation in \eqref{eq:observable distributions} yields
\begin{align}\label{eq:kappa0 and observable distribution}
    \kappa_0 := \inf\{\kappa > 0 : F^* \in \Fc_{\kappa}\}.
\end{align}
Following \citet{Donoho1988}, we define the functional
\begin{align}\label{eq:J functional}
    J(F) := \inf\{\kappa > 0 : F \in \Fc_{\kappa}\}
\end{align}
and its lower envelope
\begin{align}\label{eq:J envelope}
    J(F;\eta) := \inf\{J(Q) : \KS(Q,F) \le \eta\}, \qquad \eta \ge 0.
\end{align}
Here, $\KS(F_1,F_2) := \sup_{m \in \mathbb{Z}_+}|F_1(m)-F_2(m)|$ denotes the Kolmogorov-Smirnov distance for two distribution functions $F_1$ and $F_2$ on $\mathbb{Z}_+$. We note that $F^*$ can be well-estimated by $\F_n$, the empirical distribution function of $X_1,\ldots,X_n$, under the Kolmogorov-Smirnov distance \citep{Massart1990}. Since $J(F^*)=\kappa_0$ by \eqref{eq:kappa0 and observable distribution}, we have the following relation:
\begin{align}\label{eq:poisson-envelope}
    \KS(\F_n,F^*) \le \eta
    \quad \Longrightarrow \quad
    \kappa_0 = J(F^*) \ge J(\F_n;\eta) \ge J(F^*;2\eta).
\end{align}
Motivated by \eqref{eq:poisson-envelope}, we define
\begin{align}\label{eq:poisson neighborhood procedure}
    \hat\kappa_0 := \inf\{\kappa > 0 : \delta_n(\kappa) \le \eta_n\}, \qquad \delta_n(\kappa) &:= \inf_{H \in \Pc((0,\infty))} \KS(F_{H,\kappa}, \F_n)
\end{align}
for a given positive sequence $\eta_n \to 0$. Note that $\delta_n(\kappa)$ is nonincreasing in $\kappa$. The following theorem shows that $\hat\kappa_0$ in \eqref{eq:poisson neighborhood procedure} converges almost surely to $\kappa_0$ in \eqref{eq:kappa0} for an appropriate choice of $\eta_n$. We provide its proof in Appendix~\ref{app:proof of thm:poisson-kappa-consistency} by establishing a lower semi-continuity of the functional $J$ at $F^*$. 

\begin{theorem}\label{thm:poisson-kappa-consistency}
    Suppose that $\eta_n \to 0$ and $\KS(\F_n, F^*) \le \eta_n$ for all sufficiently large $n$ almost surely. Then, $\hat\kappa_0$ in \eqref{eq:poisson neighborhood procedure} converges to $\kappa_0$ in \eqref{eq:kappa0} almost surely. For example, if $\eta_n = C\sqrt{(\log n)/n}$ for some $C > 1/\sqrt{2}$, then $\hat\kappa_0 \to \kappa_0$ almost surely as $n \to \infty$.
\end{theorem}

\section{Simulation}\label{sec:simulation}
In this section, we study the finite-sample performance of the EB marginal coverage sets through simulation studies. We compare our procedure with the confidence intervals of \citet{Garwood1936} and the robust EB confidence intervals (REBCIs) of \citet{Armstrong2022}. In each setting, we generate $n = 1000$ independent observations and repeat the experiment $100$ times at nominal coverage level $1-\beta = 0.95$. We consider the following four priors under the Poisson mixture model \eqref{eq:poisson model}: (i) $G^* = \tfrac{1}{2}\,\mathrm{Gamma}(2,2) + \tfrac{1}{2}\,\mathrm{Gamma}(2,4)$; (ii) $G^* = \tfrac{3}{4}\mathrm{Gamma}(3,1) + \tfrac{1}{4}\mathrm{Gamma}(3,10)$; (iii) $G^* = \mathrm{Lognormal}(0,1)$; and (iv) $G^* = \tfrac{1}{2}\mathrm{IG}(1,1) + \tfrac{1}{2}\mathrm{IG}(3,9)$, where $\mathrm{IG}$ denotes the inverse-Gaussian distribution. Here, we choose a unimodal Gamma mixture prior in Setting (i) and a bimodal Gamma mixture prior in Setting (ii). The lognormal prior in Setting (iii) and the inverse-Gaussian mixture prior in Setting (iv) lie outside our Gamma mixture prior class. To construct the EB marginal coverage sets (OPTs) \eqref{eq:poisson estimated optimal set}, we use the practical procedure for computing $\hat\kappa_0$ described in Appendix~\ref{app:computation-neighborhood-procedure}.

Table~\ref{tab:poisson-simulations} reports the average coverage and length of each procedure over the $100$ replications. Across all four settings, our procedure yields substantially shorter average length than the competing methods while maintaining coverage close to the nominal 0.95 level. Note that REBCIs yield a non-trivial reduction in average length relative to the Garwood intervals, but they remain more conservative than our EB marginal coverage sets. Our procedure also performs well under misspecified priors, suggesting the flexibility of the Gamma mixture prior class.

\begin{table}[t]
\centering
\caption{Average coverage probabilities and lengths of the coverage sets for each setting. Standard deviations are reported in parentheses.}
\vspace{3mm}
\begin{tabular}{c|c|c|c|c|c}
Method & Metric & (i) & (ii) & (iii) & (iv) \\
\hline
\multirow{2}{*}{OPT} & Cov. & 0.951 (0.012) & 0.951 (0.008) & 0.947 (0.008) & 0.949 (0.008) \\
 & Len. & 1.654 (0.087) & 4.022 (0.111) & 3.073 (0.121) & 3.625 (0.103) \\
\hline
\multirow{2}{*}{Garwood} & Cov. & 0.989 (0.003) & 0.983 (0.004) & 0.985 (0.003) & 0.984 (0.004) \\
 & Len. & 4.920 (0.042) & 6.842 (0.084) & 5.974 (0.080) & 6.466 (0.082) \\
\hline
\multirow{2}{*}{REBCI} & Cov. & 0.992 (0.003) & 0.983 (0.005) & 0.983 (0.004) & 0.984 (0.004) \\
 & Len. & 3.686 (0.074) & 6.186 (0.106) & 4.905 (0.147) & 5.717 (0.104) \\
\end{tabular}
\label{tab:poisson-simulations}
\end{table}

\section{Real data example}\label{sec:real-data-example}
We illustrate our procedure using the 2017--18 and 2018--19 National Hockey League (NHL) skater goal dataset available at \href{https://www.hockey-reference.com/}{\texttt{https://www.hockey-reference.com/}}. We use the data for $n = 745$ players who played in both the 2017--18 and 2018--19 seasons. Let $X_i$ denote the total number of goals scored by the $i$th player in the 2017--18 season, and assume that the Poisson mixture model \eqref{eq:poisson model} holds. \citet{Jana2025} use this dataset to illustrate the prediction performance of minimum-distance estimators (including the classical NPMLE) for the 2018--19 season based on 2017--18 season data. In Appendix~\ref{app:real-data-example-prediction}, we show that our procedure based on the smooth NPMLE also achieves prediction performance comparable to that of the classical NPMLE and other minimum-distance estimators. 

In this section, we focus on showing that the smooth NPMLE can also be used to make inference for $\{\theta_i\}_{i=1}^{n}$ based on the 2017--18 season data. Applying the neighborhood procedure in Section~\ref{sec:identifiability-poisson} and Appendix~\ref{app:computation-neighborhood-procedure}, we select $\hat{\kappa}_0 = 1.5$. We then compute the smooth NPMLE and construct the EB marginal coverage sets; see Figure~\ref{fig:nhl-real-data} for an illustration. Our smooth NPMLE indicates that the prior distribution is concentrated on low scoring rates and has a long right tail corresponding to high-scoring players. Our EB marginal coverage sets are shorter, on average, than Garwood intervals and REBCIs. Averaged over the observed skaters, the proposed sets have length $9.493$, compared with $11.830$ for Garwood intervals and $11.403$ for REBCIs. Thus, our procedure yields reductions of $19.8\%$ and $16.8\%$ in average length relative to Garwood intervals and REBCIs, respectively.

\begin{figure}[h]
\centering
\begin{minipage}[t]{0.32\textwidth}
\centering
\includegraphics[width=\linewidth]{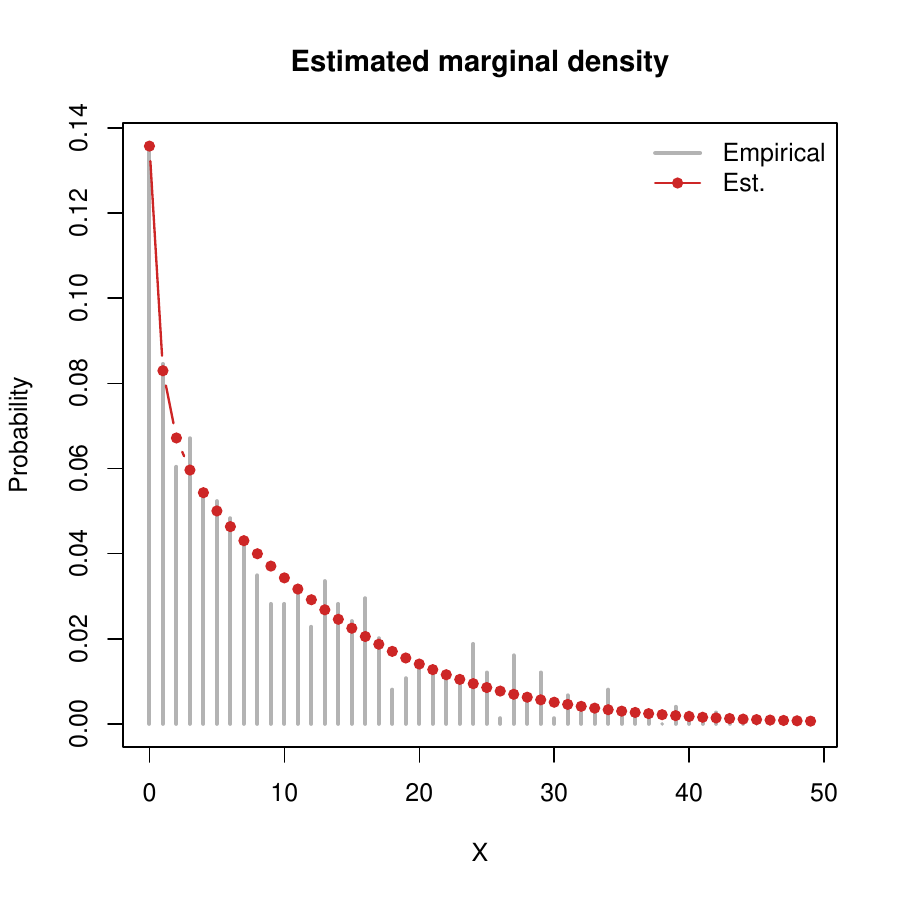}
\end{minipage}
\hfill
\begin{minipage}[t]{0.32\textwidth}
\centering
\includegraphics[width=\linewidth]{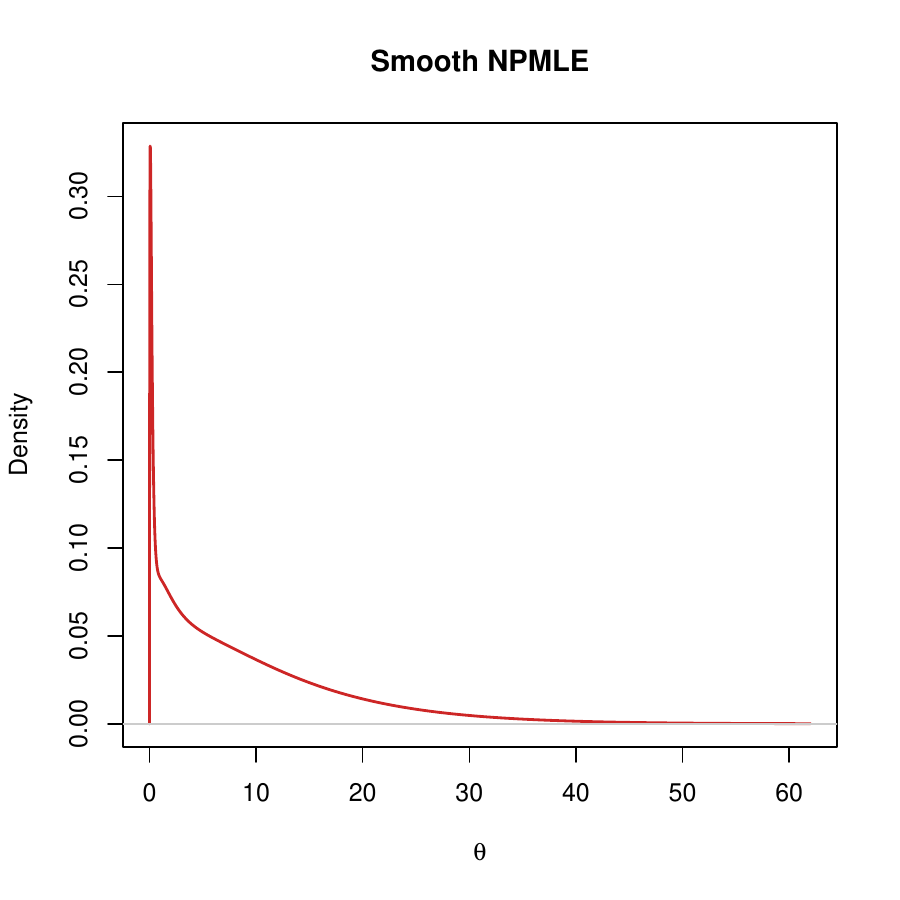}
\end{minipage}
\hfill
\begin{minipage}[t]{0.32\textwidth}
\centering
\includegraphics[width=\linewidth]{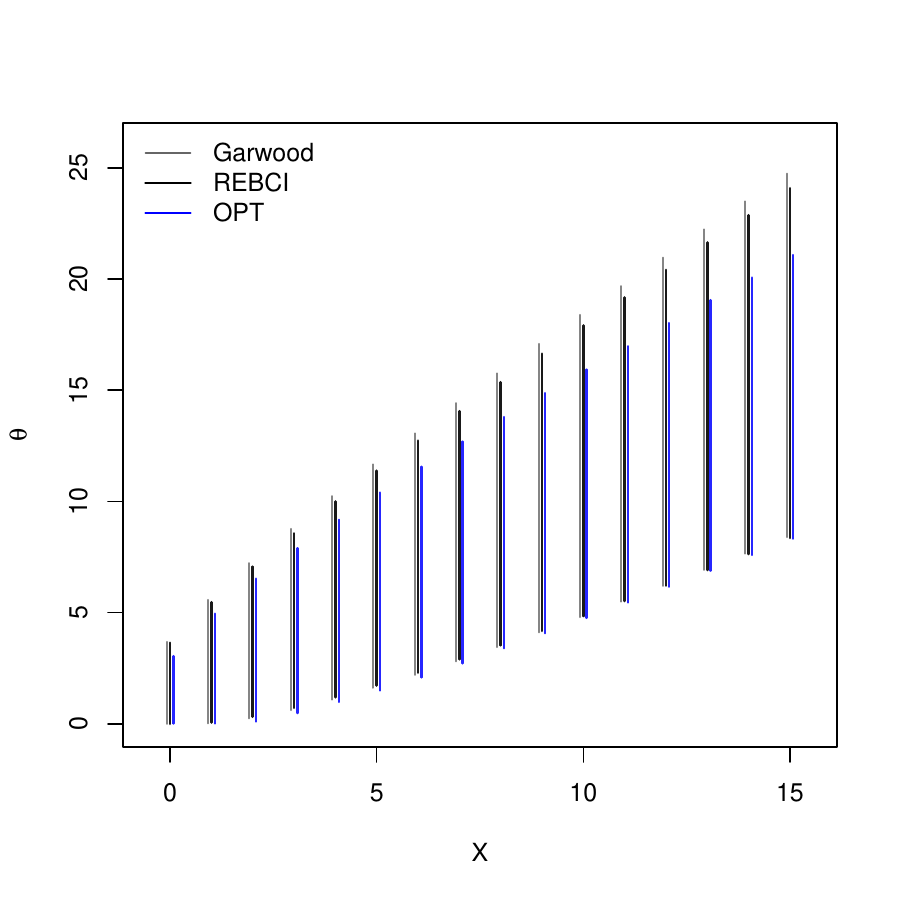}
\end{minipage}
\caption{EB analysis of the 2017--18 NHL skater goal dataset: estimated marginal density, smooth NPMLE and 95\% EB marginal coverage sets with Garwood intervals and REBCIs.}
\label{fig:nhl-real-data}
\end{figure}

\section*{Acknowledgements}
The author would like to thank Bodhisattva Sen for helpful discussions.
\clearpage
\bibliographystyle{plainnat}
\bibliography{bib}

\clearpage
\phantomsection\label{supplementary-material}
\bigskip
\begin{center}
{\LARGE \bf Appendix}
\end{center}
\appendix
\section{Discussion on the polynomial convergence rate of the smooth NPMLE}\label{app:illustration-smooth-NPMLE}
\paragraph*{Illustration of the polynomial rate}
First, we provide simulation results to illustrate that the total variation distance between $g_{\npmle}$ and $g_{\trueprior}$ decays at a polynomial rate. In Figure~\ref{fig:illustration-TV}, the true priors are (i) $\tfrac{1}{2}\mathrm{Gamma}(1,1) + \tfrac{1}{2}\mathrm{Gamma}(1,2)$, (ii) $\tfrac{1}{2}\mathrm{Gamma}(2,1) + \tfrac{1}{2}\mathrm{Gamma}(2,2)$, (iii) $\tfrac{1}{2}\mathrm{Gamma}(2,2) +  \tfrac{1}{2}\mathrm{Gamma}(2,4)$, and (iv) $ \tfrac{1}{2}\mathrm{Gamma}(4,2) +  \tfrac{1}{2}\mathrm{Gamma}(4,4)$. The empirical slopes are approximately $-0.291$, $-0.272$, $-0.260$, and $-0.215$ for panels (i), (ii), (iii), and (iv), respectively. These results indicate that the smooth NPMLE converges at a polynomial rate up to logarithmic factors. Although this rate is slower than the nearly parametric rate $n^{-1/2}$, it is clearly faster than a logarithmic rate. This behavior is consistent with Theorem~\ref{thm:poisson convergence of smooth NPMLE}. Here, in these simulations the NPMLE is computed without the technical support constraint $[L,U]$ assumed in Theorem~\ref{thm:poisson convergence of smooth NPMLE}.

\begin{figure}[htbp]
    \centering
    \includegraphics[width=\textwidth]{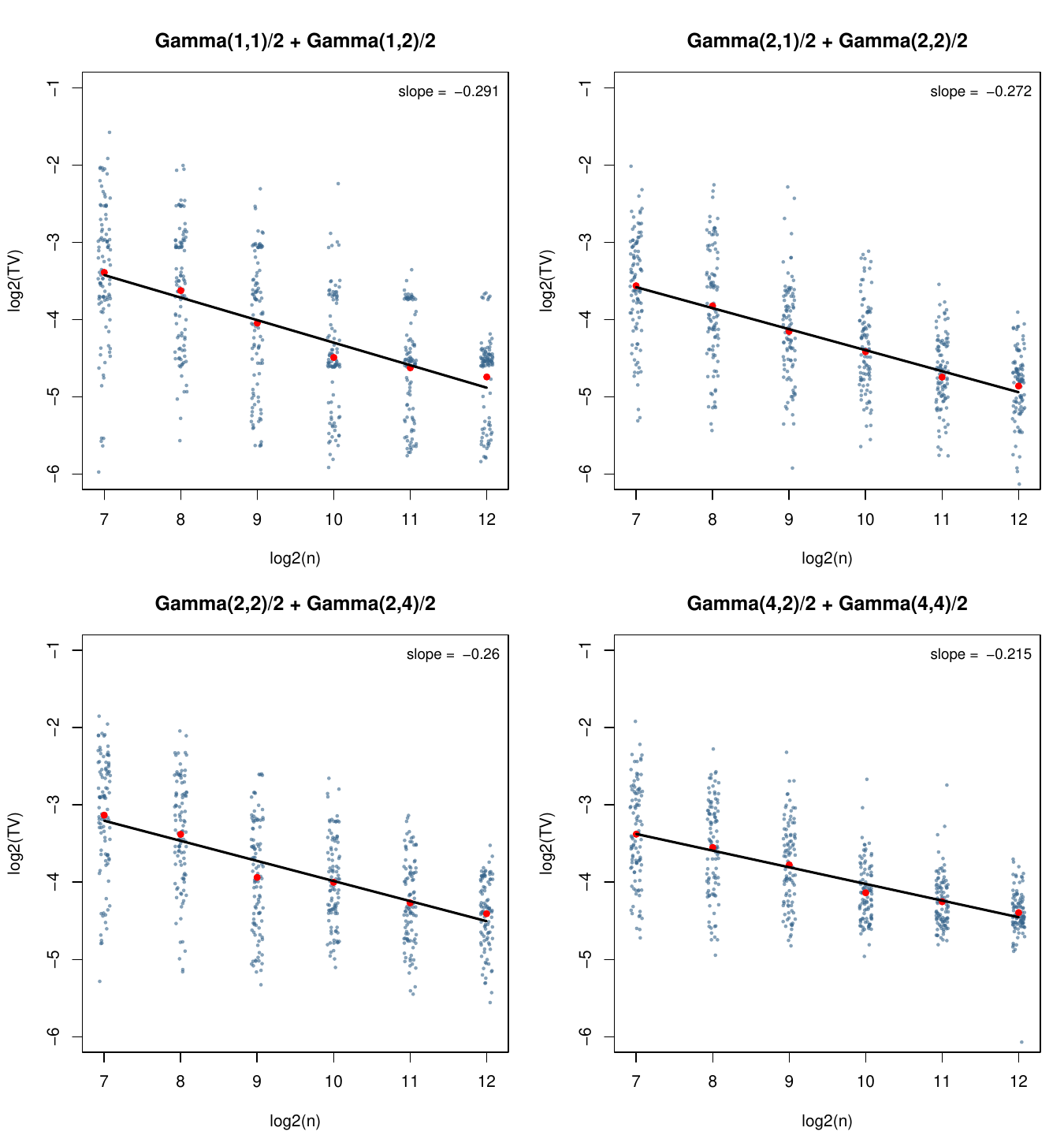}
    \caption{Total variation distance between the smooth NPMLE $g_{\npmle}$ in \eqref{eq:poisson estimated prior density} and the true prior density $g_{\trueprior}$ in \eqref{eq:poisson true prior density}. We consider sample sizes $n = 2^7, 2^8, \ldots, 2^{12}$ and perform $100$ Monte Carlo replications for each sample size. The red dots denote the averages of the log-scaled distances over the 100 replications, and the regression line is fitted using all simulated values.}
    \label{fig:illustration-TV}
\end{figure}

\paragraph*{Comparison with the Gaussian-smoothed optimal transport framework}
We now compare our polynomial convergence rate with results obtained under the Gaussian-smoothed optimal transport (GOT) framework of \citet{Goldfeld2020}. Recent works \citet{Han2023,teh2023} establish polynomial convergence rates for the classical NPMLE under the GOT distance. We note that this approach differs from our smooth NPMLE: they compute the NPMLE under the Poisson mixture model \eqref{eq:poisson model} and then compare Gaussian-smoothed distributions. In contrast, we compute the NPMLE $\hat{H}_n$ under the negative binomial mixture model \eqref{eq:negative binomial model} and then obtain the Gamma mixture prior density $g_{\npmle}$ in \eqref{eq:poisson estimated prior density}. 

\paragraph*{Proof sketch}
Our proof is nevertheless inspired by \citet{Han2023}: we reduce prior density estimation to marginal density estimation through an approximation argument; see Appendix~\ref{app:proof of thm:poisson convergence of smooth NPMLE} for a full proof. However, in our smooth NPMLE setting, this reduction takes a different form. After reparameterizing the model, we use the dual characterization of total variation to reduce the problem to controlling certain linear functionals of the estimated and true mixing laws. The key step is Lemma~\ref{lem:laguerre-approximation}, which approximates the resulting kernel transform uniformly by finite linear combinations of the negative binomial basis functions; its proof in turn relies on Lemma~\ref{lem:qpsi-polynomial-approximation}, the Laguerre derivative bound in Lemma~\ref{lem:poisson-laguerre-derivative}, and coefficient bounds from \citet{Han2023} based on \citep{Timan2014,DeVore1976}. Combining this approximation with the marginal density bound from Theorem~\ref{thm:poisson marginal density estimation} yields the polynomial rate. Similarly to \citet{Han2023}, our rate $\alpha_* \in (0,1/2)$ in Theorem~\ref{thm:poisson convergence of smooth NPMLE} may not be optimal, and establishing a matching lower bound is beyond the scope of this paper. However, we conjecture that the optimal rate depends on $(L,U)$ analogously to how the optimal rate depends on the problem parameters in the Gaussian case; see Theorem 2.2 of \citet{kim2026empiricalbayesestimationinference}.

\section{Computation of the neighborhood procedure for Poisson mixture models}\label{app:computation-neighborhood-procedure}
In this section, we provide a practical computation strategy for $\hat\kappa_0$ defined in \eqref{eq:poisson neighborhood procedure}. First, suppose that the neighborhood size $\eta_n$ is given. Then, for each $\kappa > 0$, we need to compute
\begin{align*}
    \delta_n(\kappa) = \inf_{H \in \Pc((0,\infty) )} \KS(F_{H,\kappa}, \F_n)
\end{align*}
where $F_{H,\kappa}$ is defined in \eqref{eq:observable distributions} and $\F_n$ is the empirical distribution function of $X_1, \ldots, X_n$. To compute it numerically, we approximate $H$ by a discrete distribution supported on a finite grid $\Lambda_M = \{\lambda_1,\ldots,\lambda_M\} \subset (0,\infty)$.  Writing
\begin{align*}
    H_w := \sum_{j=1}^M w_j\delta_{\lambda_j}, \qquad w_j \ge 0, \qquad \sum_{j=1}^M w_j = 1,
\end{align*}
the induced count distribution function is
\begin{align*}
    F_{H_w,\kappa}(m) = \sum_{j=1}^M w_j \sum_{x=0}^m \frac{\Gamma(x+\kappa)}{\Gamma(\kappa)x!}\left(\frac{1}{1+\lambda_j}\right)^x \left(\frac{\lambda_j}{1+\lambda_j}\right)^{\kappa}, \qquad m \in \mathbb Z_+.
\end{align*}
Let $X_{(n)} = \max_{1\le i \le n} X_i$. Since $\F_n(m)=1$ for every $m \ge X_{(n)}$ and $F_{H_w,\kappa}$ is nondecreasing in $m$, the Kolmogorov-Smirnov distance need only be evaluated over the finite set $\{0,1,\ldots,X_{(n)}\}$. Therefore, the finite-grid approximation
\begin{align*}
    \delta_{n,M}(\kappa) = \inf_{w \in \Delta_M}\max_{0 \le m \le X_{(n)}} |F_{H_w,\kappa}(m)-\F_n(m)|
\end{align*}
with $\Delta_M = \{w \in [0,1]^M : \sum_{j=1}^M w_j = 1\}$ can be computed by the linear program
\begin{equation}\label{eq:poisson-neighborhood-lp}
\begin{aligned}
    \underset{t \ge 0,\; w \in \Delta_M}{\text{minimize}} \quad & t \\
    \text{subject to} \quad
    & F_{H_w,\kappa}(m)-\F_n(m) \le t, \qquad m=0,\ldots,X_{(n)},\\
    & F_{H_w,\kappa}(m)-\F_n(m) \ge -t, \qquad m=0,\ldots,X_{(n)}.
\end{aligned}
\end{equation}
Once $\delta_{n,M}(\kappa)$ has been evaluated over a grid $\kappa_1 < \cdots < \kappa_L$ with a sufficiently large $\kappa_L$, we define the grid-based approximation
\begin{align*}
    \hat\kappa_{0,M} = \min\{\kappa_{\ell} : \delta_{n,M}(\kappa_{\ell}) \le \eta_n\}.
\end{align*}
For each fixed $\kappa$, the quantity $\delta_{n,M}(\kappa)$ is obtained by solving a linear program in $(t,w)$, and $\hat\kappa_{0,M}$ is then found by a one-dimensional search over the grid $\kappa_1 < \cdots < \kappa_L$. This makes the neighborhood procedure straightforward to implement with standard linear programming software.

It remains to discuss how to choose the neighborhood radius in practice. The consistency result in Theorem~\ref{thm:poisson-kappa-consistency} assumes that $\eta_n$ is a given deterministic sequence. In applications, however, we may treat $\eta$ as a tuning parameter and select it based on the data. We recommend $K$-fold cross-validation over a candidate grid of values of $\eta$. For each candidate $\eta$ and each fold, we compute $\hat\kappa_{0,M}$ from the training sample, fit the corresponding smooth NPMLE under $\kappa_*=\hat\kappa_{0,M}$, and evaluate the average held-out log-likelihood of the validation counts under the fitted marginal model. We then select the value of $\eta$ that maximizes the average validation log-likelihood across the $K$ folds. Finally, using this selected value, we compute $\hat\kappa_{0,M}$ on the full sample.

For the simulation studies in Section~\ref{sec:simulation} and the real data example in Section~\ref{sec:real-data-example}, we first select the neighborhood radius $\eta$ by five-fold cross-validation. We then evaluate $\delta_{n,M}(\kappa)$ on the full sample over the grid $\{0.1, 0.2, 0.3, \ldots, 6\}$ and set $\hat\kappa_{0,M}$ to be the smallest grid value of $\kappa$ such that $\delta_{n,M}(\kappa) \le \eta$.

\section{Comparison of prediction results for the NHL Data}\label{app:real-data-example-prediction}
In this section, we compare prediction results for the 2018--19 NHL season based on the 2017--18 data discussed in Section~\ref{sec:real-data-example}, using the minimum-distance estimators of \citet{Jana2025} and our smooth NPMLE. Let $X_i$ and $Z_i$ denote the goals scored by player $i$ in the 2017--18 and 2018--19 seasons, respectively. As noted by \citet{Jana2025}, if, conditionally on $\theta_i$, the past and future counts are independent and both follow the Poisson model with mean $\theta_i$, then the predictor of $Z_i$ that minimizes mean squared error is $\E[Z_i \mid X_i] = \E[\theta_i \mid X_i]$. Therefore, EB estimators of the posterior mean \eqref{eq:poisson posterior mean} with $\npmle$ can be used directly for future-season prediction.

On the sample of $745$ players, our implementation of the classical NPMLE under \eqref{eq:poisson model} yields a root mean squared error (RMSE) of $6.036$. This matches the NPMLE results reported in Table 1 of \citet{Jana2025}. Our smooth NPMLE yields an RMSE of $6.048$, which is essentially identical to that of the classical NPMLE. These values are also comparable to those of the minimum Hellinger and minimum $\chi^2$ estimators reported in Table 1 of \citet{Jana2025}, whose RMSEs are $6.02$ and $6.05$, respectively. Figure~\ref{fig:nhl-prediction-comparison} compares the observed future-season totals with the predictions from the classical NPMLE and the smooth NPMLE. Most predictions are nearly identical for values below $40$, and they differ only beyond $40$ where relatively few data points are available. This shows that our smooth NPMLE can achieve prediction performance comparable to that of the classical NPMLE and other minimum-distance estimators, while also providing uncertainty quantification. 

\begin{figure}[htbp]
\centering
\includegraphics[width=0.52\textwidth]{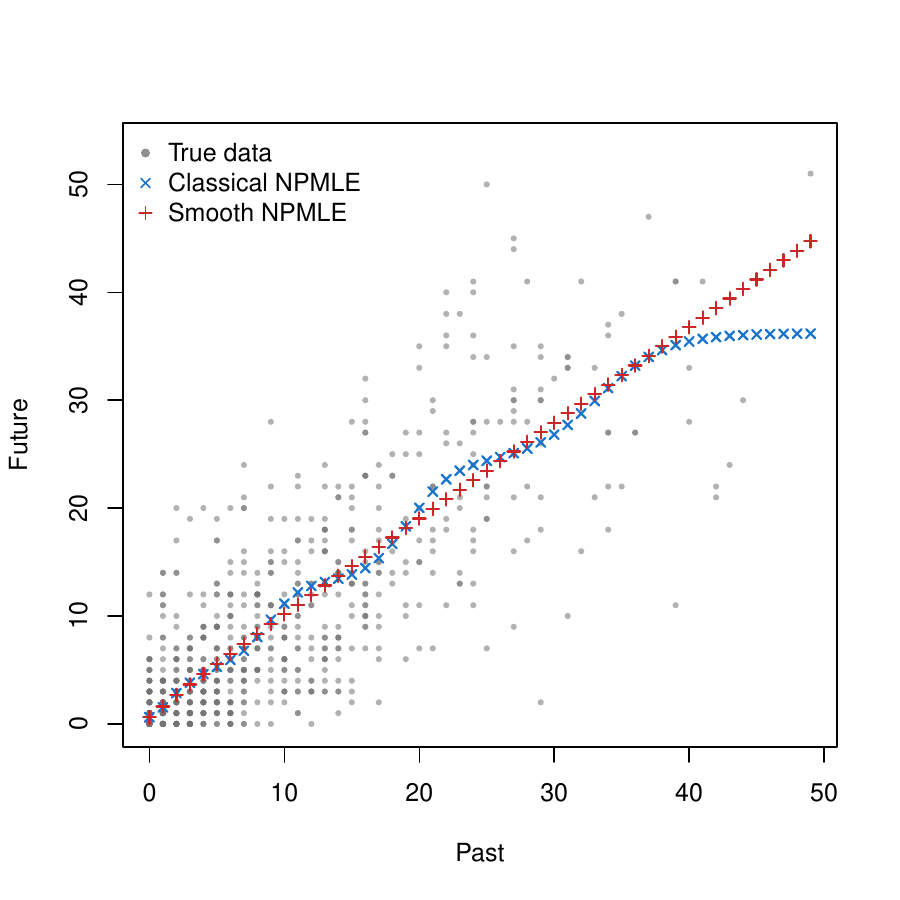}
\caption{Prediction of 2018--19 NHL goals from the 2017--18 season data. The gray points show the observed future-season goal totals against the past-season totals, while the blue X markers and red + markers show the posterior mean predictions under the classical NPMLE and the smooth NPMLE, respectively.}
\label{fig:nhl-prediction-comparison}
\end{figure}

\section{Proofs for Section~\ref{sec:poisson details}}\label{app:proof for sec:poisson details}
\subsection{Proof of Theorem~\ref{thm:poisson marginal density estimation}}\label{app:proof of thm:poisson marginal density estimation}
\begin{proof}[Proof of Theorem~\ref{thm:poisson marginal density estimation}]
To show \eqref{eq:poisson density estimation}, we adapt the proof of Theorem 2 in \citet{Jana2025}. We first note that the NPMLE $\npmle$ in \eqref{eq:poisson NPMLE} is the minimum-distance estimator with respect to the Kullback-Leibler (KL) divergence, i.e., $\npmle = \argmin_{H \in \Pc((0,\infty] )} \mathrm{KL}(p_n ||f_{H})$ where $p_n$ is the empirical distribution of $X_1, \ldots, X_n$ and $f_{H}$ is defined in \eqref{eq:poisson marginal density}. For any probability mass functions $p_1$ and $p_2$ on $\mathbb{Z}_+$, it holds that $\Hellinger^2(p_1, p_2) \le \mathrm{KL}(p_1 ||p_2) \le \chi^2(p_1, p_2)$. Then we have
\begin{align*}
    \Hellinger^2(f_{\npmle}, f_{\trueprior}) &\le (\Hellinger(f_{\npmle}, p_n) + \Hellinger(p_n, f_{\trueprior}))^2 \le 2(\Hellinger^2(p_{n}, f_{\npmle}) + \Hellinger^2(p_n, f_{\trueprior})) \\
    &\le 2(\mathrm{KL}(p_n || f_{\npmle}) + \mathrm{KL}(p_n|| f_{\trueprior})) \le 4\mathrm{KL}(p_n || f_{\trueprior}).
\end{align*}
Let $X_{(n)} = \max_{i} X_i$. For some $K > 0$ that will be defined later, we bound the KL divergence by the chi-squared divergence:
\begin{align*}
    &\E_{\trueprior}[\mathrm{KL}(p_n || f_{\trueprior})\1v(X_{(n)} < K)] \le \E_{\trueprior}[\chi^2(p_n || f_{\trueprior})\1v(X_{(n)} < K)] \\
    &= \sum_{x\ge0} \frac{\E_{\trueprior}[(p_n(x) - f_{\trueprior}(x))^2 \1v(X_{(n)} < K)]}{f_{\trueprior}(x)} \\
    &= \sum_{x < K} \frac{\E_{\trueprior}[(p_n(x) - f_{\trueprior}(x))^2 \1v(X_{(n)} < K)]}{f_{\trueprior}(x)} + \sum_{x \ge K} f_{\trueprior}(x) \P_{\trueprior}(X_{(n)} < K)
\end{align*}
where the last equality is due to the fact that $p_n(x) = 0$ for any $x \ge K$ on the event $\{X_{(n)} < K \}$. Note that $\E_{\trueprior}[p_n(x)] = f_{\trueprior}(x)$ and $\E_{\trueprior}[(p_n(x) - f_{\trueprior}(x))^2] = \Var_{\trueprior}[p_n(x)] = f_{\trueprior}(x)(1-f_{\trueprior}(x))/n$. Hence, we get
\begin{align*}
    \E_{\trueprior}[\mathrm{KL}(p_n || f_{\trueprior})\1v(X_{(n)} < K)] &\le \frac{1}{n}\sum_{x < K} (1- f_{\trueprior}(x)) + \sum_{x \ge K} f_{\trueprior}(x)\left(1- \sum_{x \ge K}f_{\trueprior}(x) \right)^n \\
    &\le \frac{K}{n} + \rho_K(\trueprior)(1-\rho_K(\trueprior))^n \le \frac{K}{n} + \rho_K(\trueprior)
\end{align*}
where we write
\begin{align}\label{eq:tail probabiliity}
    \rho_K(\trueprior) = \P_{\trueprior}(X \ge K) = \sum_{x=K}^{\infty} f_{\trueprior}(x).
\end{align}
Now, note that $\E_{\trueprior}[\Hellinger^2(f_{\trueprior}, f_{\npmle})\1v(X_{(n)} \ge K)] \le 2 \P_{\trueprior}(X_{(n)} \ge K) \le 2n\rho_K(\trueprior)$ by the union bound. Hence, we have
\begin{align*}
    \E_{\trueprior}[\Hellinger^2(f_{\trueprior}, f_{\npmle})] &\le \E_{\trueprior}[\Hellinger^2(f_{\trueprior}, f_{\npmle})\1v(X_{(n)} < K)] + \E_{\trueprior}[\Hellinger^2(f_{\trueprior}, f_{\npmle})\1v(X_{(n)} \ge K)] \\
    &\le 4\E_{\trueprior}[\mathrm{KL}(p_n || f_{\trueprior})\1v(X_{(n)} < K)] + 2n\rho_K(\trueprior) \\
    &\le \frac{4K}{n} + (4+ 2n)\rho_K(\trueprior).
\end{align*}
We need to pick $K$ so that the tail probability of the negative binomial mixture $\rho_K(\trueprior)$ satisfies $\rho_K(\trueprior) \lesssim n^{-2}$. Let $h = L/ (L+1)$, $t = -\log (1-h)/2$ and $A_h = \left(\frac{h}{1- \sqrt{1-h}} \right)^{\kappa_*} \in (0,\infty)$. We pick
\begin{align*}
    K = \Big\lceil \frac{5 \log n + \log A_h}{t} \Big\rceil.
\end{align*}
Then, by Lemma~\ref{lem:NB tail probability} below, we have $\rho_K(\trueprior) \le n^{-5}$ and thus we get \eqref{eq:poisson density estimation}.
\end{proof}

\begin{lemma}\label{lem:NB tail probability}
Fix $\kappa > 0$ and $h \in (0,1)$. Suppose that  $G \in \Pc([h,1))$ and $\{X_i\}_{i=1}^{n} \overset{iid}\sim f_{G}$ where $f_{G}$ is the negative binomial mixture density
\begin{align*}
    f_{G}(x)\;:=\; \int \frac{\Gamma(x+\kappa)}{x!\Gamma(\kappa)} p^{\kappa}(1-p)^x \diff G(p), \quad x =0, 1, \ldots.
\end{align*}
Let $t = -\log(1-h) / 2 > 0$ and $ A_h = \left(\frac{h}{1- \sqrt{1-h}} \right)^{\kappa} \in (0, \infty)$. For any $a > 1$, let $K = \big\lceil \frac{a \log n + \log A_{h}}{t} \big\rceil$. Then for any integer $s \ge 0$, 
\begin{align*}
    \P_{G}(X_1 \ge K + s) \le \frac{e^{-ts}}{n^a}.
\end{align*}
\end{lemma}
\begin{proof}[Proof of Lemma~\ref{lem:NB tail probability}]
    By the generalized binomial theorem, we have
    \begin{align*}
        (1-z)^{-\kappa} = \sum_{x=0}^{\infty} \frac{\Gamma(x+\kappa)}{x!\Gamma(\kappa)}z^x, \qquad \forall  \kappa > 0, \;\abs{z} < 1.
    \end{align*}
    Take $z = (1-p)r$ with some $r \ge 0$ so that $z < 1$. Then
    \begin{align*}
        (1-(1-p)r)^{-\kappa} = \sum_{x=0}^{\infty} \frac{\Gamma(x+\kappa)}{x!\Gamma(\kappa)}(1-p)^xr^x.
    \end{align*}
    By multiplying both sides by $p^{\kappa}$, we obtain
    \begin{align}\label{eq:expectation of r^X given p}
        \E[r^X \mid p] = \sum_{x=0}^{\infty} \frac{\Gamma(x+\kappa)}{x!\Gamma(\kappa)} p^{\kappa}(1-p)^x r^x = \left(\frac{p}{1-(1-p)r}\right)^{\kappa}
    \end{align}
    where $X\mid p \sim \mathrm{NB}(\kappa,p)$.  Now, set $r=e^t$ with $t = -\log(1-h)/2$ in \eqref{eq:expectation of r^X given p}. Then note that
    $(1-p)r = (1-p)e^t \le (1-h)e^t = \sqrt{1-h} < 1$ for any $p \in [h, 1)$ and
    \begin{align*}
        \E[e^{tX} \mid p] = \left(\frac{p}{1-(1-p)e^t} \right)^{\kappa}.
    \end{align*}
    For $p \in [h,1)$, note that
    \begin{align*}
        \frac{\diff}{\diff p}\log \E[e^{tX} \mid p] = \kappa\left(\frac{1}{p} - \frac{e^t}{1-(1-p)e^t} \right) = \frac{\kappa(1-e^t)}{p(1-(1-p)e^t)} < 0.
    \end{align*}
    Thus $p \mapsto \E[e^{tX} \mid p]$ is decreasing on $[h,1)$ and
    \begin{align*}
        \E_{G}[e^{tX}] = \int_{h}^{1} \E[e^{tX} \mid p] \diff G(p) \le \left(\frac{h}{1-(1-h)e^{t}} \right)^{\kappa} = A_h.
    \end{align*}
    By Markov's inequality, for any $m \ge 0$, it holds that
    \begin{align*}
        \P_{G}(X \ge m) = \P_G(e^{tX} \ge e^{tm}) \le e^{-tm} \E_{G}[e^{tX}] \le A_{h}e^{-tm}.
    \end{align*}
    Recall that we defined $K = \big\lceil \frac{a \log n + \log A_{h}}{t} \big\rceil$ for $a > 1$. Then we have $A_{h}e^{-tK} \le n^{-a}$. Set $m = K+s$. Then, it holds that
    \begin{align}\label{eq:tail inequality NB}
        \P_{G}(X \ge K+s) \le A_{h}e^{-t(K+s)} = A_{h}e^{-tK}e^{-ts} \le n^{-a}e^{-ts}.
    \end{align}
    This completes the proof.
\end{proof}

\subsection{Proof of Theorem~\ref{thm:poisson mean estimation}}\label{app:proof of thm:poisson mean estimation}
\begin{proof}[Proof of Theorem~\ref{thm:poisson mean estimation}]
To show \eqref{eq:poisson regret bound}, we adapt the proof of Theorem 3 in \citet{Jana2025}. We first note that, since $\trueprior \in \Pc([L,\infty))$, the true prior $G^*$ with density $g_{\trueprior}$ in \eqref{eq:poisson true prior density} belongs to $\mathrm{SubE}(s) := \{G : G([t, \infty)) \le 2 e^{-t/s}, \forall t > 0\}$ for some $s=s(L,\kappa_*)>0$. This is because, for $\theta \sim g_\trueprior$ (i.e., $\theta \mid \lambda \sim \mathrm{Gamma}(\kappa_*, \lambda)$ and $\lambda \sim \trueprior$), it holds that
\begin{align*}
    \P_{\trueprior}(\theta \ge t) &\le e^{-\gamma t} \E_{\trueprior}[e^{\gamma \theta}] = e^{-\gamma t} \E_{\trueprior}[\E[e^{\gamma \theta } \mid \lambda]] \\
    &= e^{-\gamma t} \E_\trueprior\left[\left(\frac{\lambda}{\lambda - \gamma} \right)^{\kappa_*}\right] \le e^{-\gamma t }\left(\frac{L}{L-\gamma} \right)^{\kappa_*}
\end{align*}
for all $\gamma  \in (0, L)$. Then $\gamma = L(1-2^{-1/\kappa_*})$ yields $\P_{\trueprior}(\theta \ge t) \le 2e^{-\gamma t }$ and we can set $s = \gamma^{-1}$ for $\mathrm{SubE}(s)$. Now, we choose 
\begin{align}\label{eq:param-i}
    h = 4s \log n,\quad K = \max \left\{1, \frac{2}{\log\left(1+\frac{1}{2s} \right)} \right\}\log n, \quad M = 1 \vee 48s^4.
\end{align}
Since $G^* \in \mathrm{SubE}(s)$, we have
\begin{align*}
    \E_{\trueprior}[\theta^4] = 4\int_0^\infty t^3 \P_{\trueprior}(\theta \ge t)\,\diff t \le 8\int_0^\infty t^3 e^{-t/s}\,\diff t \le 48s^4 \le M.
\end{align*}
Moreover, Lemma D.3 of \citet{Jana2025} gives
\begin{align}\label{eq:param-ii}
    G^*((h,\infty)) \le \frac{2}{n^4}, \quad \rho_K(G^*) \le \frac{3}{n^2}.
\end{align}
Next, note that, on the event $\{X_{(n)}>0\}$, $f_{\npmle}$ is fully supported on $\mathbb Z_+$. Hence, $\hat\theta_{\npmle}(X)$ is well defined. Moreover, under \eqref{eq:poisson hierarchical model}, for any finite $\lambda$,
\begin{align*}
    \theta_i \mid X_i = x, \lambda_i = \lambda \sim \mathrm{Gamma}(x + \kappa_*, \lambda+1)
\end{align*}
and we have
\begin{align*}
    \E[\theta_i \mid X_i = x, \lambda_i = \lambda ] = \frac{x +\kappa_*}{\lambda +1} \le x + \kappa_*.
\end{align*}
Since this bound holds for every finite $\lambda$, the same bound holds after averaging over the posterior distribution of $\lambda_i$ given $X_i=x$. Hence, it follows that
\begin{align*}
    \hat\theta_{\npmle}(X) = (X+1)\frac{f_{\npmle}(X+1)}{f_{\npmle}(X)} \le X + \kappa_*.
\end{align*}
When $X_{(n)}=\max_i X_i=0$ so that $\npmle=\delta_{\infty}$, we define $\hat\theta_{\npmle}(x)=0$ for all $x \in \mathbb Z_+$ and can handle this case separately. By Lemma~\ref{lem:poisson regret bound unbounded estimator} below with \eqref{eq:param-i} and \eqref{eq:param-ii}, it holds that
\begin{align*}
    \E_{\trueprior}\left[\left(\hat\theta_{\trueprior}(X) - \hat\theta_{\npmle}(X)\right)^2\right] &\lesssim_{\kappa_*,L} (\log n)^2\left(\E_{\trueprior}\left[\Hellinger^2(f_{\trueprior}, f_{\npmle}) \right] + \frac{1}{n^4}\right) + \frac{(\log n)^2}{n} + \frac{1}{n^2} \\
    &\lesssim_{\kappa_*,L} \frac{(\log n)^3}{n}
\end{align*}
where in the last inequality we used \eqref{eq:poisson density estimation} in Theorem~\ref{thm:poisson marginal density estimation}. This completes the proof.
\end{proof}

\begin{lemma}\label{lem:poisson regret bound unbounded estimator}
    Suppose that $X \mid \theta \sim \mathrm{Poi}(\theta)$ and $\theta \sim G \in \Pc([0,\infty) )$. Let $G$ be a distribution such that $\E_{G}[\theta^4] \le M$ for some constant $M \ge 1$. Let $\hat{G}$ be any (possibly unbounded) estimator of $G$ that is independent of $(\theta,X)$. Let $\hat\theta_{\tilde{G}}(x) = \frac{(x+1)f_{\tilde{G}}(x+1)}{f_{\tilde{G}}(x)}$ for any distribution $\tilde{G} \in \Pc([0,\infty) )$ where $f_{\tilde{G}}$ denotes the marginal pmf of $X$ induced by $\tilde{G}$. Suppose that, for some $\kappa > 0$, we have $\hat\theta_{\hat{G}}(x) \le x + \kappa$ for all $x \in \mathbb Z_+$ almost surely. Then for any $h > 0$ with $G([0,h]) > 1/2$ and any $K \ge 1$, 
    \begin{align*}
        &\E_{G}\left[\left(\hat\theta_G(X) - \hat\theta_{\hat{G}}(X)\right)^2\right]\\
        &\le (12(h^2 + (K+\kappa)^2) + 48(h+K+\kappa)K)\left(\E_{G}\left[\Hellinger^2(f_{G}, f_{\hat{G}})\right]+4G((h,\infty)) \right)\\
        &\quad + 2\sqrt{2}\left((2h+\kappa)^2 +\sqrt{h + 3h^2}\right) \sqrt{\rho_K(G)}
        + 2(1+6\sqrt{15})\sqrt{(\kappa^4 + M)G((h,\infty))}
    \end{align*}
    where $\rho_K(G) = \P_{G}(X \ge K)$.
\end{lemma}

\begin{proof}[Proof of Lemma~\ref{lem:poisson regret bound unbounded estimator}]
    The proof of this lemma closely follows that of Lemma 4 in \citet{Jana2025} but with one important difference: Lemma 4 in \citet{Jana2025} considers $\hat{G}$ to be supported on a compact interval, but here we consider a possibly unbounded $\hat{G}$ while assuming $\hat\theta_{\hat{G}}(x) \le x+\kappa$. First, write $G_h(A) = G(A \cap [0,h]) / G([0,h])$ for any measurable set $A \subseteq \Real$. Then
    \begin{align*}
        &\E_{G}\left[(\hat\theta_{\hat{G}}(X) - \theta)^2 \big| \hat{G} \right] \le \E_{G}\left[(\hat\theta_{\hat{G}}(X) - \theta)^2 \1v(\theta \le h)\big| \hat{G} \right] + \E_{G}\left[(\hat\theta_{\hat{G}}(X) - \theta)^2 \1v(\theta > h)\big| \hat{G} \right]  \\
        &\le \E_{G_h}\left[(\hat\theta_{\hat{G}}(X) - \theta)^2 \big| \hat{G} \right] + \sqrt{\E_{G}\left[(\hat\theta_{\hat{G}}(X) - \theta)^4  \big| \hat{G}\right] G((h,\infty)) }.
    \end{align*}
    Since $0 \le \hat\theta_{\hat{G}}(X) \le X+\kappa$, we have $|\hat\theta_{\hat{G}}(X) - \theta| \le X + \kappa + \theta \le |X-\theta| + 2\theta + \kappa$. Hence, using $(x+y+z)^4 \le 27(x^4+y^4+z^4)$,
    \begin{align*}
        \E_{G}\left[(\hat\theta_{\hat{G}}(X) - \theta)^4 \big| \hat{G}\right]
        &\le 27\left(\E_{G}\left[(X-\theta)^4\right] + 16\E_{G}\left[\theta^4\right] + \kappa^4\right) \le 27(20M+\kappa^4),
    \end{align*}
    where we used $\E_{G}[(X-\theta)^4] = \E_{G}\big[\E[(X-\theta)^4 \mid \theta]\big] = \E_{G}[\theta] + 3\E_{G}[\theta^2] \le 4M$ since $M \ge 1$. Therefore, we have
    \begin{align*}
        \E_{G}\left[(\hat\theta_{\hat{G}}(X) - \theta)^2 \big| \hat{G} \right]
        \le \E_{G_h}\left[(\hat\theta_{\hat{G}}(X) - \theta)^2 \big| \hat{G} \right] + \sqrt{27(20M+\kappa^4) G((h,\infty)) }.
    \end{align*}
    Using this bound together with equation (C.2) in \citet{Jana2025}, we obtain
    \begin{align*}
        \E_{G}\left[\left(\hat\theta_G(X) - \hat\theta_{\hat{G}}(X)\right)^2 \Big| \hat{G} \right]
        &\le \E_{G_h}\left[ \left(\hat\theta_{G_h}(X) - \hat\theta_{\hat{G}}(X)\right)^2 \Big| \hat{G} \right] \\
        &+ \frac{(1+6\sqrt{15})\sqrt{(\kappa^4 + M)G((h,\infty))}}{G([0,h])}
    \end{align*}
    where we used $\sqrt{27(20M+\kappa^4)} \le 6\sqrt{15}\sqrt{M+\kappa^4}$.
    Fix $K \ge 1$. Using $\hat\theta_{G_h}(x) \le h$ and $\hat\theta_{\hat{G}}(x) \le x + \kappa$, we have
    \begin{align*}
        &\E_{G_h}\left[\left(\hat\theta_{G_h}(X) - \hat\theta_{\hat{G}}(X)\right)^2 \1v(X \le K-1)\Big| \hat{G} \right] \\
        &= \sum_{x=0}^{K-1}(x+1)^2f_{G_h}(x)\left(\frac{f_{G_h}(x+1)}{f_{G_h}(x)} - \frac{f_{\hat{G}}(x+1)}{f_{\hat{G}}(x)}\right)^2 \\
        &\le (6(h^2 + (K+\kappa)^2) + 24(h+K+\kappa)K)\Hellinger^2(f_{G_h}, f_{\hat{G}})
    \end{align*}
    following the arguments on page 35 of \citet{Jana2025}. Moreover, again using $\hat\theta_{G_h}(x) \le h$ and $\hat\theta_{\hat{G}}(x) \le x+ \kappa$, we have
    \begin{align*}
        &\E_{G_h}\left[\left(\hat\theta_{G_h}(X) - \hat\theta_{\hat{G}}(X)\right)^2 \1v(X \ge K)\Big| \hat{G} \right] \le \E_{G_h}[(h+X+\kappa)^2\1v(X \ge K)] \\
        &\le 2(2h+\kappa)^2 \rho_K(G_h) + 2\E_{G_h}[(X-\theta)^2\1v(X\ge K)] \\
        &\le 2(2h+\kappa)^2 \rho_K(G_h) +2\sqrt{\E_{G_h}\left[(X - \theta)^4\right] \rho_K(G_h)} \\
        &\le 2\left((2h+\kappa)^2  +\sqrt{h + 3h^2}\right)\sqrt{\rho_K(G_h)}
    \end{align*}
    where the last inequality holds since $\E_{G_h}[(X-\theta)^4] = \E_{G_h}[\E[(X-\theta)^4 \mid \theta]] = \E_{G_h}[\theta] + 3\E_{G_h}[\theta^2] \le h + 3h^2$ and $\rho_K(G_h) \le 1$.
    Combining all of the above, we obtain
    \begin{align*}
        \E_{G}\left[\left(\hat\theta_G(X) - \hat\theta_{\hat{G}}(X)\right)^2 \Big| \hat{G} \right]  &\le (6(h^2 + (K+\kappa)^2) + 24(h+K+\kappa)K)\Hellinger^2(f_{G_h}, f_{\hat{G}}) \\
        &+ 2\left((2h+\kappa)^2 +\sqrt{h + 3h^2}\right)\sqrt{\rho_K(G_h)} \\
        &+\frac{(1+6\sqrt{15})\sqrt{(\kappa^4 + M)G((h,\infty))}}{G([0,h])}.
    \end{align*}
    Now, note that $\Hellinger^2(f_{G_h}, f_{\hat{G}}) \le 2(\Hellinger^2(f_{G_h}, f_{G}) + \Hellinger^2(f_{G}, f_{\hat{G}}))$ and
    \begin{align*}
        \Hellinger^2(f_{G_h}, f_{G}) \le 2 \mathrm{TV}(f_{G_h}, f_{G}) \le 2 \mathrm{TV}(G_{h}, G) = 2G((h,\infty)) \le 4G((h,\infty)).
    \end{align*}
    Hence, using $\rho_K(G_h) \le \frac{\rho_K(G)}{G([0,h])}$ and $G([0,h])>1/2$, we get
    \begin{align*}
        &\E_{G}\left[\left(\hat\theta_G(X) - \hat\theta_{\hat{G}}(X)\right)^2 \right]  \\
        &\le (12(h^2 + (K+\kappa)^2) + 48(h+K+\kappa)K)\left(\E_{G}\left[\Hellinger^2(f_{G}, f_{\hat{G}})\right]+4G((h,\infty)) \right)\\
        &\quad + 2\left((2h+\kappa)^2 +\sqrt{h + 3h^2}\right) \sqrt{\frac{\rho_K(G)}{G([0,h])}}
        + \frac{(1+6\sqrt{15})\sqrt{(\kappa^4 + M)G((h,\infty))}}{G([0,h])} \\
        &\le (12(h^2 + (K+\kappa)^2) + 48(h+K+\kappa)K)\left(\E_{G}\left[\Hellinger^2(f_{G}, f_{\hat{G}})\right]+4G((h,\infty)) \right)\\
        &\quad+ 2\sqrt{2}\left((2h+\kappa)^2 +\sqrt{h + 3h^2}\right) \sqrt{\rho_K(G)}
        + 2(1+6\sqrt{15})\sqrt{(\kappa^4 + M)G((h,\infty))}.
    \end{align*}
    This completes the proof.
\end{proof}

\subsection{Proof of Theorem~\ref{thm:poisson convergence of smooth NPMLE}}\label{app:proof of thm:poisson convergence of smooth NPMLE}
Throughout this subsection, we write $\kappa=\kappa_*$. In order to prove Theorem~\ref{thm:poisson convergence of smooth NPMLE}, it would be easier to reparameterize the model as follows. For $H\in \Pc([L,U])$, let $F_H$ be the pushforward of $H$ under $\lambda\mapsto z=(1+\lambda)^{-1}$. Then, we can write
\begin{align}\label{eq:reparametrized f}
    f_H(x)=\int_I \phi_x(z)\,\diff F_H(z), \qquad \phi_x(z):=\frac{\Gamma(x+\kappa)}{x!\Gamma(\kappa)}z^x(1-z)^\kappa
\end{align}
where $I := [a, b]$ with $a := (1+U)^{-1}$ and $b := (1+L)^{-1}$. Moreover, we can write
\begin{align}\label{eq:reparametrized g}
    g_H(\theta)=\int_I \Phi_\theta(z)\,\diff F_H(z),
    \qquad
    \Phi_\theta(z):=\frac{\theta^{\kappa-1}}{\Gamma(\kappa)}
    \left(\frac{1-z}{z}\right)^\kappa e^{-\theta(1-z)/z}.
\end{align}
A high-level proof sketch is given in Appendix~\ref{app:illustration-smooth-NPMLE}; we now turn to the full argument.

\begin{proof}[Proof of Theorem~\ref{thm:poisson convergence of smooth NPMLE}]
By the dual characterization of total variation, we can write
\begin{align*}
    \mathrm{TV}(g_{\npmle},g_{\trueprior}) = \frac{1}{2} \sup_{\|\psi\|_\infty\le 1} \left| \int_0^\infty \psi(\theta)\bigl(g_{\npmle}(\theta)-g_{\trueprior}(\theta)\bigr)\,\diff\theta \right|.
\end{align*}
Fix a bounded measurable function $\psi : (0,\infty) \to \Real$ with $\|\psi\|_\infty\le 1$. Then it holds that
\begin{align*}
    \int_0^\infty \psi(\theta)\bigl(g_{\npmle}(\theta)-g_{\trueprior}(\theta)\bigr)\,\diff\theta = \int_I T_\psi(z)\,\diff F_{\npmle}(z) -\int_I T_\psi(z)\, \diff F_{\trueprior}(z)
\end{align*}
where
\begin{align}\label{eq:T_psi}
    T_{\psi}(z) =\int_0^\infty \psi(\theta)\Phi_\theta(z)\,\diff\theta.
\end{align}
with $\Phi_{\theta}(z)$ defined in \eqref{eq:reparametrized g}. Here, we can exchange the order of integration by Fubini's theorem since $|\psi(\theta)| \le 1$ and $\Phi_\theta(z)\ge 0$ with $\int_0^\infty \Phi_\theta(z)\,\diff\theta = 1$ for every $z\in I$. 

We proceed by approximating the linear functional $T_\psi$ uniformly over $\|\psi\|_\infty\le 1$ by a finite linear combination of the basis functions $\{\phi_x\}_{x=0}^m$. By Lemma~\ref{lem:laguerre-approximation} below, for every integer $m\ge 2$, there exist coefficients $a_{\psi,0}^{(m)},\dots,a_{\psi,m}^{(m)} \in \Real$ such that
\begin{align}\label{eq:part-i}
    \sup_{z\in I} \left| T_\psi(z)-\sum_{x=0}^m a_{\psi,x}^{(m)}\phi_x(z) \right| \le C\rho^m
\end{align}
where $\phi_x(z)$ is defined in \eqref{eq:reparametrized f} and
\begin{align}\label{eq:part-ii}
    \sum_{x=0}^m |a_{\psi,x}^{(m)}| \le C(m+1)^{(1-\kappa)_+}B^m
\end{align}
for some constant $C>0$ depending only on $(\kappa,L,U)$, and some $\rho \in (0,1)$ and $B > 1$ depending only on $(L,U)$.  Therefore, it holds that
\begin{align*}
    &\left|\int_I T_\psi(z)\,\diff F_{\npmle}(z) -\int_I T_\psi(z)\, \diff F_{\trueprior}(z)\right| \\
    &\overset{(a)}{\le} 2C\rho^m + \sum_{x=0}^m |a_{\psi,x}^{(m)}| \left|\int_I \phi_x(z)\,\diff F_{\npmle}(z) - \int_{I} \phi_{x}(z) \diff F_{\trueprior}(z)\right|\\
    &\overset{(b)}= 2C\rho^m + \sum_{x=0}^m |a_{\psi,x}^{(m)}|\,|f_{\npmle}(x)-f_{\trueprior}(x)|\\
    &\overset{(c)}{\le} 2C\rho^m + C(m+1)^{(1-\kappa)_+}B^m
    \|f_{\npmle}-f_{\trueprior}\|_{1}.
\end{align*}
Here, in (a), we used \eqref{eq:part-i}. In (b), we used the identity in \eqref{eq:reparametrized f}. Also, in (c), we used \eqref{eq:part-ii}. Taking the supremum over $\|\psi\|_\infty\le 1$ gives
\begin{align}\label{eq:TV-bound}
    \mathrm{TV}(g_{\npmle},g_{\trueprior}) \le C\rho^m + C(m+1)^{(1-\kappa)_+}B^m \|f_{\npmle}-f_{\trueprior}\|_{1}.
\end{align}
Since $\|p-q \|_{1} \le 2\sqrt{2}\Hellinger(p,q)$, we have
\begin{align*}
    \E_{\trueprior}[\|f_{\npmle} - f_{\trueprior} \|_{1}] \le 2\sqrt{2}\left(\E_{\trueprior}[\Hellinger^2(f_{\npmle}, f_{\trueprior}) ]\right)^{1/2} \lesssim_{\kappa, L} \left( \frac{\log n}{n}\right)^{1/2} 
\end{align*}
by \eqref{eq:poisson density estimation} in Theorem~\ref{thm:poisson marginal density estimation}, whose proof applies unchanged when the optimization domain $\Pc((0,\infty] )$ is replaced by any subclass containing $\trueprior$; in particular, under the present assumptions, it applies to $\Pc([L,U])$. 

We now balance the approximation term $\rho^m$ and the estimation term $B^m(\log n/n)^{1/2}$ by an appropriate choice of $m=m_n$. We choose
\begin{align*}
    t_n:=\frac{\log\sqrt{n/\log n}}{\log B+\log(1/\rho)}, \qquad m_n:=\lfloor t_n\rfloor\vee 1.
\end{align*}
For all sufficiently large $n$, we have $m_n\ge 2$. Since $m_n\ge t_n-1$ and $\rho\in(0,1)$, we have
\begin{align*}
    \rho^{m_n} \le \rho^{t_n-1} = \rho^{-1}\left(\frac{\log n}{n}\right)^{\alpha_*}
\end{align*}
where we define
\begin{align}\label{eq:alphastar}
    \alpha_*&:=\frac{\log(1/\rho)}{2\{\log B+\log(1/\rho)\}}\in(0,1/2)
\end{align}
with $\rho \in (0,1)$ and $B > 1$ depending only on $( L, U)$. Also, since $m_n\le t_n$ and $B>1$, we have
\begin{align*}
    B^{m_n}\left(\frac{\log n}{n}\right)^{1/2} \le B^{t_n}\left(\frac{\log n}{n}\right)^{1/2} = \left(\frac{\log n}{n}\right)^{\alpha_*}.
\end{align*}
Finally, we have  $m_n+1\lesssim \log n$. Combining the above results with \eqref{eq:TV-bound}, we have
\begin{align*}
    \E_{\trueprior}\bigl[\mathrm{TV}(g_{\npmle},g_{\trueprior})\bigr] \lesssim_{\kappa, L, U} n^{-\alpha_*}(\log n)^{(1-\kappa)_+ + \alpha_*}.
\end{align*}
This completes the proof.
\end{proof}

\begin{lemma}\label{lem:laguerre-approximation}
Let $\Ac = \{\psi:[0,\infty)\to\mathbb R:\psi \text{ is measurable and }\|\psi\|_\infty\le 1\}$. For $x\in\mathbb Z_+$ and $z\in I$, let $\phi_x(z)$ be defined as in \eqref{eq:reparametrized f}. For $\psi\in\Ac$ and $z\in I$, let $T_\psi(z)$ be defined as in \eqref{eq:T_psi}. Then there exists a constant $C>0$ depending only on $(\kappa,L,U)$ and constants $\rho\in(0,1)$ and $B>1$ depending only on $(L,U)$ such that, for every integer $m\ge2$ and every $\psi\in\Ac$, there exist coefficients $a_{\psi,0}^{(m)},\dots,a_{\psi,m}^{(m)}$ satisfying
\begin{align}\label{eq:approximation error}
    \sup_{z\in I}\left|T_\psi(z)-\sum_{x=0}^m a_{\psi,x}^{(m)}\phi_x(z)\right| \le C\rho^m
\end{align}
and
\begin{align}\label{eq:coefficient bound}
    \sum_{x=0}^m |a_{\psi,x}^{(m)}| \le C(m+1)^{(1-\kappa)_+}B^m.
\end{align}
\end{lemma}

\begin{proof}[Proof of Lemma~\ref{lem:laguerre-approximation}]

Define
\begin{align}\label{eq:z0 and Delta}
    z_0:=\frac{a+b}{2}, \qquad \Delta:=\frac{b-a}{2}
\end{align}
where $a < b$ are the endpoints of the interval $I$, which contains the support of the pushforward measure $F_H$ (see \eqref{eq:reparametrized f} and \eqref{eq:reparametrized g}). By \eqref{eq:rho} and \eqref{eq:d coefficient bound} in Lemma~\ref{lem:qpsi-polynomial-approximation} below, for every integer $m\ge 2$ and every $\psi\in\Ac$, there exists a polynomial of degree at most $m$
\begin{align*}
    p_{m,\psi}(u)=\sum_{j=0}^m d_{\psi,j}^{(m)}u^j
\end{align*}
such that
\begin{align}\label{eq:R approximation bound}
    \sup_{u \in [-1,1]}\abs{R_\psi(z_0+\Delta u)-p_{m,\psi}(u)} \lesssim_{\kappa, L, U} \rho^m
\end{align}
where $ R_\psi(z)=(1-z)^{-\kappa}T_\psi(z)$ with $T_{\psi}(z)$ defined in \eqref{eq:T_psi} and
\begin{align}\label{eq:d coefficient bound cite}
    |d_{\psi,j}^{(m)}| \lesssim_{\kappa, L, U} \frac{m^j}{j!}, \qquad j=0,\dots,m.
\end{align}
Here, $\rho \in (0,1)$ is a constant depending only on $(L,U)$. 

We now convert the above polynomial approximation of $R_\psi(z_0+\Delta u)$ into the $\{\phi_x\}$ basis. Define
\begin{align*}
    P_{m,\psi}(z):=p_{m,\psi}\left(\frac{z-z_0}{\Delta}\right), \qquad z\in I
\end{align*}
where $z_0$ and $\Delta$ are defined in \eqref{eq:z0 and Delta}. Note that $P_{m,\psi}$ is a polynomial in $z$ of degree at most $m$, thus we can write
\begin{align}\label{eq:P polynomial}
    P_{m,\psi}(z)=\sum_{x=0}^m c_{\psi,x}^{(m)}z^x
\end{align}
where, by the binomial theorem,
\begin{align*}
    c_{\psi,x}^{(m)} = \sum_{j=x}^m d_{\psi,j}^{(m)}\Delta^{-j}\binom{j}{x}(-z_0)^{j-x}.
\end{align*}
Hence, we can bound the coefficients $c_{\psi,x}^{(m)}$ using the coefficient bounds \eqref{eq:d coefficient bound cite}:
\begin{align}\label{eq:c_psi coefficient bound}
    |c_{\psi,x}^{(m)}| \lesssim_{\kappa, L, U}
    \sum_{j=x}^m \frac{m^j}{j!}\Delta^{-j}\binom{j}{x}z_0^{j-x} =
    \frac{m^x}{x!\Delta^x}
    \sum_{j=x}^m \frac{(m z_0/\Delta)^{j-x}}{(j-x)!} \le e^{m z_0/\Delta}\frac{m^x}{x!\Delta^x}.
\end{align}

We will show \eqref{eq:approximation error} and \eqref{eq:coefficient bound} based on \eqref{eq:P polynomial} and \eqref{eq:c_psi coefficient bound}. For this, we will convert the monomial expansion of $P_{m,\psi}$ into the basis $\{\phi_x\}$ by the following rescaling:
\begin{align*}
    a_{\psi,x}^{(m)} := c_{\psi,x}^{(m)}\frac{x!\Gamma(\kappa)}{\Gamma(x+\kappa)}, \qquad x=0,\dots,m.
\end{align*}
Then
\begin{align}\label{eq:P polynomial expression}
    (1-z)^\kappa P_{m,\psi}(z)=\sum_{x=0}^m a_{\psi,x}^{(m)}\phi_x(z).
\end{align}
Therefore,
\begin{align*}
    \sum_{x=0}^m |a_{\psi,x}^{(m)}| &\lesssim_{\kappa, L, U}  e^{m z_0/\Delta} \sum_{x=0}^m \frac{m^x}{x!\Delta^x}\frac{x!\Gamma(\kappa)}{\Gamma(x+\kappa)}\\
    &\overset{(a)}{\lesssim}_{\kappa, L, U} e^{m z_0/\Delta} \sum_{x=0}^m \frac{m^x}{x!\Delta^x}(x+1)^{(1-\kappa)_+}\\
    &\le (m+1)^{(1-\kappa)_+}e^{m z_0/\Delta}
    \sum_{x=0}^m \frac{(m/\Delta)^x}{x!}\\
    &\le (m+1)^{(1-\kappa)_+}\exp\left(\frac{m(z_0+1)}{\Delta}\right)\\
    &\overset{(b)}{=} (m+1)^{(1-\kappa)_+}B^m.
\end{align*}
Here, in (a), we used the fact that there exists a constant $C_{\kappa}$ depending only on $\kappa$ such that
\begin{align*}
    \frac{x!\Gamma(\kappa)}{\Gamma(x+\kappa)} \le C_\kappa (x+1)^{(1-\kappa)_+}
\end{align*}
for any $x \in \mathbb{Z}_+$ and $\kappa >  0$. In (b), we set
\begin{align}\label{eq:beta}
    B&:=\exp\left(\frac{z_0+1}{\Delta}\right)>1.
\end{align}
This proves \eqref{eq:coefficient bound}. Finally, note that $T_\psi(z)=(1-z)^\kappa R_\psi(z)$ and $0\le (1-z)^\kappa\le 1$ on $I$. Using \eqref{eq:P polynomial expression}, we have
\begin{align*}
    \sup_{z\in I} \left| T_\psi(z)-\sum_{x=0}^m a_{\psi,x}^{(m)}\phi_x(z) \right|
    \le \sup_{z\in I}\abs{R_\psi(z)-P_{m,\psi}(z)} \lesssim_{\kappa, L, U}\rho^m
\end{align*}
by \eqref{eq:R approximation bound}. This proves \eqref{eq:approximation error}. 
\end{proof}

\begin{lemma}\label{lem:qpsi-polynomial-approximation}
Let $\Ac = \left\{\psi:[0,\infty)\to\mathbb R: \psi \text{ is measurable and } \|\psi\|_\infty\le 1\right\}$. For a measurable function $\psi\in\Ac$, let
\begin{align}\label{eq:q_psi}
    q_\psi(u):=R_\psi(z_0+\Delta u), \qquad u\in[-1,1]
\end{align}
where $R_\psi(z)=(1-z)^{-\kappa}T_\psi(z)$ with $T_\psi(z)$ defined in \eqref{eq:T_psi}. Then there exists a constant $C>0$ depending only on $(\kappa,L,U)$ and a constant $\rho\in(0,1)$ depending only on $(L,U)$ such that, for every integer $m\ge2$ and every $\psi\in\Ac$, there exists a polynomial of degree at most $m$
\begin{align*}
    p_{m,\psi}(u)=\sum_{j=0}^m d_{\psi,j}^{(m)}u^j
\end{align*}
satisfying
\begin{align}\label{eq:rho}
    \sup_{u\in[-1,1]}\abs{q_\psi(u)-p_{m,\psi}(u)} \le C\rho^m
\end{align}
and
\begin{align}\label{eq:d coefficient bound}
    |d_{\psi,j}^{(m)}| \le C\frac{m^j}{j!}, \qquad j=0,\dots,m.
\end{align}
\end{lemma}

\begin{proof}[Proof of Lemma~\ref{lem:qpsi-polynomial-approximation}]
The proof combines a derivative bound for $q_\psi$ with a polynomial approximation theorem on $[-1,1]$. By Lemma~\ref{lem:poisson-laguerre-derivative} below, there exists a constant $M>0$ depending only on $(\kappa,L,U)$ and a constant $A>0$ depending only on $(L,U)$ such that
\begin{align}\label{eq:q derivative bound}
    \sup_{u \in[-1,1]}\abs{q_\psi^{(r)}(u)} \le M A^r\Gamma(r+\kappa), \qquad r=0,1,2,\cdots.
\end{align}
Now, define
\begin{align}\label{eq:r_m and tau}
    r_m:=\max\{1,\lfloor \tau m\rfloor\}, \qquad \tau := \min\left\{ \frac{1}{4}, \frac{e}{8A}\right\}.
\end{align}
Since $\tau\le 1/4$, we have $1\le r_m<m$. By \eqref{eq:q derivative bound}, $q_\psi$ has bounded derivatives up
to order $r_m+1$ on $[-1,1]$. Then, by Lemma 5.1 of \citet{Han2023} (which follows from Theorem 6.2 of \citet{DeVore1976}), there exists a polynomial $p_{m,\psi}$ of degree at most $m$ such that
\begin{align*}
    \sup_{u \in [-1,1]} \abs{q_\psi(u)-p_{m,\psi}(u)} \le C_J m^{-r_m}\omega(q_\psi^{(r_m)},m^{-1})
\end{align*}
where $C_J$ is a universal constant and
\begin{align*}
    \omega(f^{(r)}, m^{-1}) := \sup_{u_1, u_2: |u_1 - u_2| \le m^{-1}} \abs{f^{(r)}(u_1) - f^{(r)}(u_2)}.
\end{align*}
Moreover, since $\omega(g,t)\le t\|g'\|_{\infty}$ for absolutely continuous $g$, we have
\begin{align*}
    \sup_{u \in [-1,1]} \abs{q_\psi(u)-p_{m,\psi}(u)} &\le C_J m^{-r_m-1} \sup_{u \in [-1,1]} \abs{q_\psi^{(r_m+1)}(u) } \\
    &\le C_J M m^{-r_m-1}A^{r_m+1}\Gamma(r_m+\kappa+1)
\end{align*}
by \eqref{eq:q derivative bound}. Note that, by Stirling's formula, there exists a constant $C'>0$ depending only on $\kappa$ such that
\begin{align*}
    \Gamma(r+\kappa+1)\le C' (r+1)^{\kappa+1}\left(\frac{r}{e}\right)^r, \qquad r\ge 1.
\end{align*}
Since $r_m+1\le m$, it follows that
    \begin{align*}
        \sup_{u \in [-1, 1]}\abs{q_\psi(u)-p_{m,\psi}(u)}  \le C'' m^{\kappa}\left(\frac{A r_m}{e m}\right)^{r_m}
    \end{align*}
where $C''$ depends only on $(\kappa,L,U)$. For $m < 1/\tau$, we have $r_m=1$ and $r_m \le \tau m$ may fail. Since this occurs only for finitely many $m\ge 2$, it can be absorbed into the constant. Hence, it suffices to consider $m \ge 1/\tau$ for which $r_m \le \tau m$. Because $r_m\le \tau m$ and $\tau\le e/(8A)$ in \eqref{eq:r_m and tau}, we have
\begin{align*}
    \frac{A r_m}{e m}\le \frac{A\tau}{e}\le \frac{1}{8}
\end{align*}
which yields
\begin{align*}
    \sup_{u \in [-1,1]}\abs{q_\psi(u)-p_{m,\psi}(u)} \le C'' m^{\kappa}8^{-r_m}.
\end{align*}
Since $r_m\ge \tau m-1$, we have
\begin{align*}
    m^{\kappa}8^{-r_m}\le 8\,m^{\kappa}2^{-3\tau m}
    = 8(m^{\kappa}2^{-2\tau m})2^{-\tau m}.
\end{align*}
Moreover, since $m^{\kappa}2^{-2\tau m}$ is bounded uniformly in $m\ge 2$, after enlarging the constant, we obtain
\begin{align}
    \sup_{u \in [-1,1]}\abs{q_\psi(u)-p_{m,\psi}(u) } \le C''' 2^{-\tau m} = C''' \rho^m, \qquad \rho:=2^{-\tau}\in(0,1)
\end{align}
for a constant $C'''$ depending only on $(\kappa,L,U)$. This proves \eqref{eq:rho}. Here, $\rho \in (0,1)$ depends on $\tau$ in \eqref{eq:r_m and tau}, which in turn depends on $A$ depending only on $(L,U)$ (Lemma~\ref{lem:poisson-laguerre-derivative}).

Next, let us write
\begin{align*}
    p_{m,\psi}(u)=\sum_{j=0}^m d_{\psi,j}^{(m)}u^j.
\end{align*}
Note that $q_{\psi}$ is uniformly bounded from Lemma~\ref{lem:poisson-laguerre-derivative} with $r=0$. Also, $p_{m,\psi} - q_{\psi}$ is uniformly bounded by \eqref{eq:rho}. Thus, we have
\begin{align*}
    \sup_{u \in [-1,1] }\abs{p_{m,\psi}(u)}  \le  C_0
\end{align*}
for some constant $C_0$ depending only on $(\kappa,L,U)$.
Then, by Lemma 5.5 of \citet{Han2023} (which follows from \citet{Timan2014}), we have the coefficient bound
\begin{align*}
    |d_{\psi,j}^{(m)}| \le \frac{m^j}{j!} \sup_{u \in [-1,1]}\abs{p_{m,\psi}(u)} \le C_0\frac{m^j}{j!}, \qquad j=0,\dots,m.
\end{align*}
This proves \eqref{eq:d coefficient bound}.
\end{proof}

\begin{lemma}\label{lem:poisson-laguerre-derivative}
For any measurable function $\psi:[0,\infty)\to\mathbb R$ such that $\|\psi\|_\infty\le 1$, let
\begin{align*}
    q_\psi(u):=R_\psi(z_0+\Delta u), \qquad u\in[-1,1]
\end{align*}
where $ R_\psi(z)=(1-z)^{-\kappa}T_\psi(z)$ with $T_{\psi}(z)$ defined in \eqref{eq:T_psi} and $z_0$ and $\Delta$ are defined in \eqref{eq:z0 and Delta}. Then, for every integer $r\ge 0$, we have
\begin{align*}
    \sup_{u \in [-1,1]}\abs{q_\psi^{(r)}(u)} \le a^{-\kappa}L^{-\kappa}\, A^r\, \frac{\Gamma(r+\kappa)}{\Gamma(\kappa)}.
\end{align*}
where $a = (1+U)^{-1} > 0$ and $ A = \Delta a^{-1}(1 + 1/(aL)) > 0$ is a constant depending only on $( L, U)$. 
\end{lemma}

\begin{proof}[Proof of Lemma~\ref{lem:poisson-laguerre-derivative}]
First, note that
\begin{align*}
    R_\psi(z) = \frac{1}{\Gamma(\kappa)} \int_0^\infty \psi(\theta)\theta^{\kappa-1} z^{-\kappa}e^{-\theta(1-z)/z} \,\diff\theta, \qquad z \in I.
\end{align*}
Let $L_r^{(\kappa-1)}$ denote the generalized Laguerre polynomial of degree $r$ with parameter $\kappa-1 > -1$ (see, e.g., Chapter 5 of \citet{Szego1975}). Using the series expansion of $L_r^{(\kappa-1)}$ (see, e.g., (5.1.6) of \citet{Szego1975}), we can write
\begin{align}\label{eq:Laguerre}
    L_r^{(\kappa-1)}(z) = \sum_{j=0}^r (-1)^j \binom{r+\kappa-1}{r-j}\frac{z^j}{j!}.
\end{align}
Using the generalized Laguerre identity (see, e.g., (5.1.5) of \citet{Szego1975}), we have
\begin{align*}
    \frac{\diff^r}{\diff z^r}\left[z^{-\kappa}e^{-\theta(1-z)/z}\right] = (-1)^r r!\,z^{-\kappa-r}e^{-\theta(1-z)/z}L_r^{(\kappa-1)}(\theta/z).
\end{align*}
Here, since $L_r^{(\kappa-1)}$ is a polynomial of degree $r$ and $z\in I=[a,b]$, under $\|\psi\|_\infty\le 1$, the differentiated integrand is dominated by an integrable envelope independent of $z$. Hence, differentiation under the integral sign is justified by dominated convergence. Then we have
\begin{align*}
    R_\psi^{(r)}(z) = \frac{(-1)^r r!}{\Gamma(\kappa)}z^{-\kappa-r} \int_0^\infty \psi(\theta)\theta^{\kappa-1} e^{-\theta(1-z)/z} L_r^{(\kappa-1)}(\theta/z)\,\diff\theta.
\end{align*}
Using the expression of $L_r^{(\kappa-1)}$ in \eqref{eq:Laguerre}, $\|\psi\|_\infty \le 1$ and $(1-z)/z \ge L$ for $z \in I$, we have
\begin{align*}
    |R_\psi^{(r)}(z)| &\le \frac{r!}{\Gamma(\kappa)}z^{-\kappa-r} \sum_{j=0}^r \binom{r+\kappa-1}{r-j}\frac{z^{-j}}{j!} \int_0^\infty \theta^{\kappa-1+j}e^{-L\theta}\,\diff\theta\\
    &= \frac{r!}{\Gamma(\kappa)}z^{-\kappa-r} \sum_{j=0}^r \binom{r+\kappa-1}{r-j}\frac{z^{-j}}{j!} \Gamma(\kappa+j)L^{-(\kappa+j)}\\
    &= z^{-\kappa-r}L^{-\kappa} \frac{\Gamma(r+\kappa)}{\Gamma(\kappa)}
    \sum_{j=0}^r \binom{r}{j}(z^{-1}L^{-1})^j\\
    &= z^{-\kappa-r}L^{-\kappa} \frac{\Gamma(r+\kappa)}{\Gamma(\kappa)} \left(1+\frac{1}{zL}\right)^r.
\end{align*}
Since $q_\psi^{(r)}(u)=\Delta^rR_\psi^{(r)}(z_0+\Delta u)$ and $z_0+\Delta u\in[a,b]=I$ for $u\in[-1,1]$ where $\Delta$ is defined in \eqref{eq:z0 and Delta}, it follows that
\begin{align*}
    \sup_{u \in [-1,1]} \abs{q_\psi^{(r)}(u) } &\le a^{-\kappa}L^{-\kappa} \left\{\Delta a^{-1}\left(1+\frac{1}{aL}\right)\right\}^r \frac{\Gamma(r+\kappa)}{\Gamma(\kappa)} \\
    &=a^{-\kappa}L^{-\kappa} A^r \frac{\Gamma(r+\kappa)}{\Gamma(\kappa)}
\end{align*}
where $A = \Delta a^{-1}(1 + 1/(aL)) > 0$ is a constant depending only on $(L, U)$. This completes the proof.
\end{proof}

\begin{remark}[Convergence rate]\label{rem:alphastar}
Note that $\alpha_*$ in \eqref{eq:alphastar},
\begin{align*}
    \alpha_*=\frac{\log(1/\rho)}{2\{\log B+\log(1/\rho)\}}\in(0,1/2),
\end{align*}
depends on $\rho$ and $B$, where $\rho=2^{-\tau}$ with $\tau$ defined in \eqref{eq:r_m and tau}, and $B$ is defined in \eqref{eq:beta}. Moreover, these can be expressed in terms of $(L,U)$ only:
\begin{align*}
    \rho &= 2^{-\tau}, \qquad \tau = \min\left\{\frac14,\frac{e}{8A}\right\} = \min\left\{\frac14,\frac{eL(1+L)}{4(U-L)(1+L+U)}\right\},
\end{align*}
and
\begin{align*}
    B &= \exp\left(\frac{z_0+1}{\Delta}\right) = \exp\left(\frac{4+3L+3U+2LU}{U-L}\right).
\end{align*}
We can see that, as $L \to 0$ or $U \to \infty$, we have $\alpha_*\to 0$ and the leading part of the convergence rate $n^{-\alpha_*}$ becomes worse. This is consistent with the increasing difficulty of approximation when the support approaches $0$ or becomes unbounded.
\end{remark}

\subsection{Proof of Theorem~\ref{thm:poisson posterior density convergence}}\label{app:proof of thm:poisson posterior density convergence}
\begin{proof}[Proof of Theorem~\ref{thm:poisson posterior density convergence}]
    The claim follows by combining Theorems~\ref{thm:poisson marginal density estimation} and \ref{thm:poisson convergence of smooth NPMLE}. Indeed,
    \begin{align*}
        \mathrm{wTV}(\pi_{\npmle},\pi_{\trueprior})
        &= \frac{1}{2}\sum_{x=0}^{\infty}\int_0^{\infty} \left| p_{\theta}(x)\frac{g_{\npmle}(\theta)}{f_{\npmle}(x)} - p_{\theta}(x)\frac{g_{\trueprior}(\theta)}{f_{\trueprior}(x)} \right|f_{\trueprior}(x)\diff\theta \\
        &\le \frac{1}{2}\sum_{x=0}^{\infty}\int_0^{\infty}p_{\theta}(x)|g_{\npmle}(\theta)-g_{\trueprior}(\theta)|\diff\theta \\
        &\quad + \frac{1}{2}\sum_{x=0}^{\infty}\int_0^{\infty}p_{\theta}(x)g_{\npmle}(\theta)\left|\frac{f_{\trueprior}(x)}{f_{\npmle}(x)}-1\right|\diff\theta \\
        &= \mathrm{TV}(g_{\npmle},g_{\trueprior}) + \mathrm{TV}(f_{\npmle},f_{\trueprior})
    \end{align*}
    where the last identity uses $\sum_{x=0}^{\infty} p_{\theta}(x) = 1$ for any $\theta \ge 0$ and \eqref{eq:poisson marginal density}. Therefore,
    \begin{align*}
        \E_{\trueprior}\left[\mathrm{wTV}(\pi_{\npmle},\pi_{\trueprior})\right]
        \le \E_{\trueprior}\left[\mathrm{TV}(g_{\npmle},g_{\trueprior})\right] + \E_{\trueprior}\left[\mathrm{TV}(f_{\npmle},f_{\trueprior})\right].
    \end{align*}
    By Theorem~\ref{thm:poisson convergence of smooth NPMLE}, we have
    \begin{align*}
        \E_{\trueprior}\left[\mathrm{TV}(g_{\npmle},g_{\trueprior})\right] \lesssim_{\kappa_*, L, U} n^{-\alpha_*} (\log n)^{(1-\kappa_*)_+ + \alpha_*}
    \end{align*}
    for some constant $\alphastar \in (0,1/2)$ depending only on $(L, U)$. Also, as noted in the proof of Theorem~\ref{thm:poisson convergence of smooth NPMLE}, the proof of Theorem~\ref{thm:poisson marginal density estimation} applies unchanged when the optimization domain is replaced by any subclass containing $\trueprior$; hence it applies here with $\Pc([L,U])$. Since $\mathrm{TV}(p,q)\le \sqrt{2}\,\Hellinger(p,q)$ for probability mass functions $p,q$ on $\mathbb{Z}_{\ge 0}$, Theorem~\ref{thm:poisson marginal density estimation} and Cauchy-Schwarz give
    \begin{align*}
        \E_{\trueprior}\left[\mathrm{TV}(f_{\npmle},f_{\trueprior})\right] \lesssim_{\kappa_*, L} \left(\frac{\log n}{n}\right)^{1/2}.
    \end{align*}
    Since $\alpha_* < 1/2$, the latter term is bounded by a constant multiple of $n^{-\alpha_*}(\log n)^{(1-\kappa_*)_+ + \alpha_*}$ and the result follows.
\end{proof}

\section{Proofs for Section~\ref{sec:OPT}}\label{app:proof for sec:OPT}
\subsection{Proof of Theorem~\ref{thm:poisson-opt-marginal}}\label{app:proof of thm:poisson-opt-marginal}
\begin{proof}[Proof of Theorem~\ref{thm:poisson-opt-marginal}]
Let $Y_i := \pi_{\trueprior}(\theta_i \mid X_i)$. By Theorem 3.1 of \citet{kim2026empiricalbayesestimationinference}, it suffices to show that $Y_i$ has interval support and no atoms. Fix $x \in \mathbb{Z}_+$ with $f_{\trueprior}(x) > 0$ and define
\begin{align*}
    h_x(\theta) := \pi_{\trueprior}(\theta \mid x), \qquad \theta > 0.
\end{align*}
By \eqref{eq:poisson true prior density} and \eqref{eq:poisson posterior density}, we have 
\begin{align}\label{eq:poisson posterior analytic form}
    h_x(\theta)
    = \frac{\theta^{x+\kappa_* -1}}{x!\Gamma(\kappa_*)f_{\trueprior}(x)}
    \int \lambda^{\kappa_*} e^{-(\lambda+1)\theta}\diff \trueprior(\lambda), \qquad \theta > 0.
\end{align}
Hence, $h_x$ is a real-analytic and non-constant function on $(0,\infty)$ and $h_x(\theta)\to 0$ as $\theta\to\infty$. Since $h_x$ is continuous on $(0,\infty)$, its image $R_x := h_x((0,\infty))$ is an interval with $0$ in its closure.

Now, fix $k \ge 0$. Since the level set $\{\theta > 0 : h_x(\theta)=k\}$ is discrete (or empty), we have
\begin{align*}
    \P_{\trueprior}(h_x(\theta_i)=k \mid X_i=x) = \int_{\{\theta > 0 : h_x(\theta)=k\}} h_x(\theta)\diff \theta = 0.
\end{align*}
Also, if $J \subset R_x$ is a non-empty open interval, then $h_x^{-1}(J)$ is a non-empty open subset of $(0,\infty)$ and
\begin{align*}
    \P_{\trueprior}(h_x(\theta_i)\in J \mid X_i=x) = \int_{h_x^{-1}(J)} h_x(\theta)\diff \theta > 0.
\end{align*}
Therefore, the conditional distribution of $Y_i = h_{X_i}(\theta_i)$ given $X_i=x$ has no atoms and interval support. Then the support of $Y_i$ is the closure of the union of the intervals $\{R_x : f_{\trueprior}(x)>0\}$, which is again an interval because each $R_x$ has $0$ in its closure.
Moreover, for every $k \ge 0$,
\begin{align*}
    \P_{\trueprior}(Y_i=k) = \sum_{x : f_{\trueprior}(x)>0} \P_{\trueprior}(h_x(\theta_i)=k \mid X_i=x)f_{\trueprior}(x) = 0.
\end{align*}
Hence, $Y_i$ has interval support and no atoms. Therefore, Theorem 3.1 of \citet{kim2026empiricalbayesestimationinference} yields that $\Ic^*$ defined in \eqref{eq:poisson oracle optimal set} satisfies \eqref{eq:poisson oracle threshold exact} and solves \eqref{eq:poisson optimal marginal set problem} up to Lebesgue-null sets.
\end{proof}

\subsection{Proof of Theorem~\ref{thm:poisson-opt-coverage}}\label{app:proof of thm:poisson-opt-coverage}
\begin{proof}[Proof of Theorem~\ref{thm:poisson-opt-coverage}]
    Conditionally on $\npmle$, let $(\tilde{X},\tilde{\theta})$ be an independent draw from the fitted hierarchical Poisson model corresponding to $\npmle$. That is, $\tilde\theta \sim g_{\npmle}$ and $\tilde{X} \mid \tilde\theta \sim \mathrm{Poi}(\tilde\theta)$. The same argument as in the proof of Theorem~\ref{thm:poisson-opt-marginal}, applied conditionally on $\npmle$, shows that the conditional distribution of $\pi_{\npmle}(\tilde\theta \mid \tilde{X})$ given $\npmle$ has no atoms. Hence, by the definition of $\hat{k}_n$ in \eqref{eq:poisson estimated threshold}, we have
    \begin{align*}
        \P_{\npmle}(\tilde\theta \in \hat{\Ic}_n(\tilde{X})\mid \npmle) = \sum_{x=0}^{\infty} \int_{0}^{\infty} \1v(\pi_{\npmle}(\theta \mid x) \ge \hat{k}_n) p_{\theta}(x) g_{\npmle}(\theta) \diff \theta = 1-\beta.
    \end{align*}
    Therefore, it holds that
    \begin{align*}
        &\left|\P_{\trueprior}(\theta \in \hat{\Ic}_n(X)\mid \npmle) - (1-\beta)\right| \\
        &= \left|\P_{\trueprior}(\theta \in \hat{\Ic}_n(X)\mid \npmle) - \P_{\npmle}(\tilde\theta \in \hat{\Ic}_n(\tilde{X})\mid \npmle)\right| \\
        &= \left| \sum_{x=0}^{\infty}\int_0^{\infty}
        \1v\big(\pi_{\npmle}(\theta \mid x) \ge \hat{k}_n\big)p_{\theta}(x)\big(g_{\trueprior}(\theta)-g_{\npmle}(\theta)\big)\diff \theta\right| \\
        &\le \sum_{x=0}^{\infty}\int_0^{\infty}
        p_{\theta}(x)\, \left|g_{\trueprior}(\theta)-g_{\npmle}(\theta)\right| \diff \theta \\
        &= \int_0^{\infty}\left(\sum_{x=0}^{\infty} p_{\theta}(x)\right) \left|g_{\trueprior}(\theta)-g_{\npmle}(\theta)\right|\diff \theta \\
        &= \int_0^{\infty}\left|g_{\trueprior}(\theta)-g_{\npmle}(\theta)\right|\diff \theta= 2\,\mathrm{TV}(g_{\npmle},g_{\trueprior}),
    \end{align*}
    where we used $\sum_{x=0}^{\infty} p_{\theta}(x)=1$ for every $\theta \ge 0$. Taking expectation and applying Theorem~\ref{thm:poisson convergence of smooth NPMLE} gives
    \begin{align*}
        &\E_{\trueprior}\left[\left|\P_{\trueprior}(\theta \in \hat{\Ic}_n(X)\mid \npmle) - (1-\beta)\right|\right]
        \\
        &\le 2\,\E_{\trueprior}\left[\mathrm{TV}(g_{\npmle},g_{\trueprior})\right]
        \lesssim_{\kappa_*,L,U}
        n^{-\alpha_*}(\log n)^{(1-\kappa_*)_+ + \alpha_*},
    \end{align*}
    which proves \eqref{eq:poisson-opt-coverage}.
\end{proof}

\section{Proofs for Section~\ref{sec:identifiability-poisson}}\label{app:proof for sec:identifiability-poisson}
\subsection{Proof of Lemma~\ref{lem:nestedness of Gamma mixture}}\label{app:proof of lem:nestedness of Gamma mixture}
\begin{proof}[Proof of Lemma~\ref{lem:nestedness of Gamma mixture}]
Fix $0<\kappa<\kappa'$ and write $\delta = \kappa'-\kappa > 0$. Let $H \in \Pc((0,\infty) )$ be such that $g_{H,\kappa} \in \Gc_{\kappa}$. For each $\beta > 0$, let $W \sim \mathrm{Beta}(\kappa,\delta)$ and $T \sim \mathrm{Gamma}(\kappa',\beta)$ be independent. By the Beta-Gamma identity, it can be easily checked that $WT \sim \mathrm{Gamma}(\kappa,\beta)$. Moreover, conditional on $W=w$, we have $WT \mid W=w \sim \mathrm{Gamma}(\kappa',\beta/w)$. Let $K_{\beta}$ denote the distribution of $\beta/W$ on $[\beta,\infty)$. Then for every $\theta > 0$, we have
\begin{align}\label{eq:nested representation}
    \frac{\beta^\kappa}{\Gamma(\kappa)}\theta^{\kappa-1}e^{-\beta\theta}
    = \int \frac{\lambda^{\kappa'}}{\Gamma(\kappa')}\theta^{\kappa'-1}e^{-\lambda \theta}\,\diff K_{\beta}(\lambda).
\end{align}

Define a probability measure $H'$ on $(0,\infty)$ by
\begin{align*}
    H'(A) := \int K_{\beta}(A)\,\diff H(\beta), \qquad A \subseteq (0,\infty) \quad \text{measurable}.
\end{align*}
Then, by Fubini's theorem and \eqref{eq:nested representation}, for every $\theta > 0$, it holds that
\begin{align*}
    g_{H,\kappa}(\theta)
    &= \int \frac{\beta^\kappa}{\Gamma(\kappa)}\theta^{\kappa-1}e^{-\beta\theta}\,\diff H(\beta) \\
    &= \int \int \frac{\lambda^{\kappa'}}{\Gamma(\kappa')}\theta^{\kappa'-1}e^{-\lambda \theta}\,\diff K_{\beta}(\lambda)\,\diff H(\beta) \\
    &= \int \frac{\lambda^{\kappa'}}{\Gamma(\kappa')}\theta^{\kappa'-1}e^{-\lambda \theta}\,\diff H'(\lambda) \\
    &= g_{H',\kappa'}(\theta).
\end{align*}
Hence, $g_{H,\kappa} \in \Gc_{\kappa'}$. Since $H$ was arbitrary, $\Gc_{\kappa} \subseteq \Gc_{\kappa'}$. This completes the proof.
\end{proof}

\subsection{The smallest shape parameter $\kappa_0$ in \eqref{eq:kappa0}}\label{app:proof of smallest shape parameter}
In this section, we prove that (i) $\kappa_0$ in \eqref{eq:kappa0} is strictly positive and (ii) the infimum in \eqref{eq:kappa0} is attained.
\begin{proof}
(i) We first show that $\kappa_0>0$. Let $f^*(x):=F^*(x)-F^*(x-1)$ denote the true marginal pmf of $X_1$, and define the negative binomial pmf
\begin{align*}
    r_{\kappa,\rho}(x):=\frac{\Gamma(x+\kappa)}{\Gamma(\kappa)x!}(1-\rho)^x\rho^\kappa
\end{align*}
with endpoint conventions $r_{\kappa,0}(x)\equiv 0$ and $r_{\kappa,1}(x)=\1v(x=0)$. Since $g_{\trueprior}$ is a density on $(0,\infty)$, we have $f^*(1)=\int_0^\infty e^{-\theta}\theta g_{\trueprior}(\theta)\,\diff\theta>0$. Now, fix $\kappa>0$ such that $F^*\in\Fc_{\kappa}$. Then, there exists $H\in\Pc((0,\infty))$ such that
\begin{align*}
    f^*(x)=\int r_{\kappa,\rho}(x)\,\diff M(\rho)
\end{align*}
where $M$ is the pushforward of $H$ under $\lambda\mapsto \rho=\lambda/(1+\lambda)\in(0,1)$. In particular, we have
\begin{align*}
    f^*(1)=\int \kappa(1-\rho)\rho^\kappa\,\diff M(\rho)\le \kappa.
\end{align*}
This implies that every $\kappa$ with $F^*\in\Fc_{\kappa}$ satisfies $\kappa\ge f^*(1)>0$. Then, by \eqref{eq:kappa0 and observable distribution}, we have $\kappa_0\ge f^*(1)>0$.

(ii) To show that $\kappa_0$ in \eqref{eq:kappa0} attains its infimum, it suffices to show that $F^* \in \Fc_{\kappa_0}$; then \eqref{eq:kappa0 and observable distribution} and identifiability of Poisson mixtures imply $g_{\trueprior}\in\Gc_{\kappa_0}$. Choose $\kappa_j\downarrow \kappa_0$ such that $F^*\in\Fc_{\kappa_j}$ for every $j$. For each $j$, let $H_j\in\Pc((0,\infty))$ induce $F^*$ and let $M_j$ be its pushforward under $\lambda\mapsto \rho=\lambda/(1+\lambda)$ which is  a probability measure on $[0,1]$. By weak compactness of $\Pc([0,1])$, after passing to a subsequence if necessary, we may assume $M_j \Rightarrow M$ for some $M\in\Pc([0,1])$. Since $\kappa_0>0$, for each fixed $x\in\mathbb Z_+$, the map $(\kappa,\rho)\mapsto r_{\kappa,\rho}(x)$ is continuous on $[\kappa_0/2, 2\kappa_0] \times [0,1]$. Hence, we have
\begin{align}\label{eq:unif conv kappa j}
    \sup_{\rho\in[0,1]}\abs{r_{\kappa_j,\rho}(x)-r_{\kappa_0,\rho}(x)}\to 0.
\end{align}
By \eqref{eq:unif conv kappa j},  we have
\begin{align*}
    \left|\int r_{\kappa_j,\rho}(x)\,\diff M_j(\rho) -\int r_{\kappa_0,\rho}(x)\,\diff M_j(\rho)\right| \le \sup_{\rho\in[0,1]}\abs{r_{\kappa_j,\rho}(x)-r_{\kappa_0,\rho}(x)} \to 0.
\end{align*}
Also, $\rho\mapsto r_{\kappa_0,\rho}(x)$ is bounded and continuous on $[0,1]$, so weak convergence yields
\begin{align}\label{eq:weak conv kappa0}
    \int r_{\kappa_0,\rho}(x)\,\diff M_j(\rho) \to \int r_{\kappa_0,\rho}(x)\,\diff M(\rho).
\end{align}
Combining \eqref{eq:unif conv kappa j} and \eqref{eq:weak conv kappa0}, we have
\begin{align}\label{eq:fstar}
    f^*(x)=\lim_{j\to\infty}\int r_{\kappa_j,\rho}(x)\,\diff M_j(\rho)
    =\int r_{\kappa_0,\rho}(x)\,\diff M(\rho),
    \qquad x\in\mathbb Z_+.
\end{align}
Summing over $x$ and using Tonelli's theorem, we obtain
\begin{align*}
    1=\sum_{x=0}^\infty f^*(x) =\int \sum_{x=0}^\infty r_{\kappa_0,\rho}(x)\,\diff M(\rho) =1-M(\{0\})
\end{align*}
and thus $M(\{0\})=0$. Now, let $\widetilde H$ be the pushforward of $M$ under $\rho\mapsto \lambda=\rho/(1-\rho)$ sending $\rho=1$ to $\lambda=\infty$. Then \eqref{eq:fstar} shows that $f^*$ is the marginal pmf induced by $(\widetilde H,\kappa_0)$, where the kernel at $\lambda=\infty$ is interpreted as $\1v(x=0)$. Let $\widetilde G$ denote the corresponding mixing distribution on $[0,\infty)$ obtained by mixing the $\mathrm{Gamma}(\kappa_0,\lambda)$ laws over $\widetilde H$ on $(0,\infty)$ and placing an atom of size $\widetilde H(\{\infty\})$ at $0$. Then it holds that
\begin{align*}
    \sum_{x=0}^{\infty} z^x f^*(x)=\int e^{-(1-z)\theta}\,\diff \widetilde G(\theta), \qquad 0\le z<1.
\end{align*}
Since $f^*$ is also induced by the true prior $G^*$, identifiability of Poisson mixtures gives $\widetilde G=G^*$. Because $G^*$ is absolutely continuous on $(0,\infty)$, it has no atom at $0$ and $\widetilde H(\{\infty\})=\widetilde G(\{0\})=0$. Therefore, $\widetilde H$ is supported on $(0,\infty)$, and thus $F^*\in\Fc_{\kappa_0}$. This completes the proof.
\end{proof}

\subsection{Proof of Theorem~\ref{thm:poisson-kappa-consistency}}\label{app:proof of thm:poisson-kappa-consistency}
\begin{proof}[Proof of Theorem~\ref{thm:poisson-kappa-consistency}]
(i) We first show that the functional $J$ in \eqref{eq:J functional} is lower semi-continuous at $F^*$. Let $F_m$ be distribution functions on $\mathbb{Z}_+$ such that $\KS(F_m,F^*) \to 0$. Set
\begin{align*}
    \ell := \liminf_{m \to \infty} J(F_m).
\end{align*}
If $\ell=\infty$, there is nothing to prove. Otherwise, fix $\epsilon > 0$. Then there are infinitely many $m$ such that $J(F_m) < \ell+\epsilon$. Since $J(F_m) < \ell+\epsilon$ for infinitely many $m$ and the classes $\{\Fc_{\kappa} : \kappa > 0\}$ are nested (see Lemma~\ref{lem:nestedness of Gamma mixture} and \eqref{eq:kappa0 and observable distribution}), we have $F_m \in \Fc_{\ell+\epsilon}$ for infinitely many $m$. Passing to a subsequence, we may write $F_m = F_{H_m,\ell+\epsilon}$ for some $H_m \in \Pc((0,\infty))$. Writing
\begin{align*}
    f_m(x) := F_m(x)-F_m(x-1), \qquad f^*(x) := F^*(x)-F^*(x-1),
\end{align*}
we have $f_m(x)\to f^*(x)$ for every $x \in \mathbb{Z}_+$. Reparameterize $\lambda$ by $\rho = \lambda/(1+\lambda) \in [0,1]$, define
\begin{align*}
    r_{\kappa,\rho}(x) := \frac{\Gamma(x+\kappa)}{\Gamma(\kappa)x!}(1-\rho)^x \rho^{\kappa},
\end{align*}
with endpoint conventions $r_{\kappa,0}(x)\equiv 0$ and $r_{\kappa,1}(x)=\1v(x=0)$, and let $M_m$ be the pushforward of $H_m$ under $\lambda \mapsto \rho$. Then we have
\begin{align*}
    f_m(x) = \int r_{\ell+\epsilon,\rho}(x)\diff M_m(\rho).
\end{align*}
By weak compactness of $\Pc([0,1])$, along a further subsequence, we have $M_m \Rightarrow M$ for some $M \in \Pc([0,1])$. Since $\rho \mapsto r_{\ell+\epsilon,\rho}(x)$ is bounded and continuous on $[0,1]$, it holds that
\begin{align*}
    f^*(x) = \lim_{m\to\infty} f_m(x) = \int r_{\ell+\epsilon,\rho}(x)\diff M(\rho).
\end{align*}
By the same argument as in Appendix~\ref{app:proof of smallest shape parameter}, this representation implies $F^* \in \Fc_{\ell+\epsilon}$. Therefore, we have $J(F^*) \le \ell+\epsilon$. Then, letting $\epsilon \to 0$ gives
\begin{align*}
    J(F^*) \le \liminf_{m \to \infty} J(F_m).
\end{align*}
This proves that the functional $J$ in \eqref{eq:J functional} is lower semi-continuous at $F^*$.

(ii) Now, let $A_n := \{\KS(\F_n,F^*) \le \eta_n\}$. Then, by assumption, $A_n$ occurs for all sufficiently large $n$ almost surely. Note that, if $\delta_n(\kappa)$ in \eqref{eq:poisson neighborhood procedure} satisfies $\delta_n(\kappa) \le \eta_n$, then for every $\epsilon > 0$ there exists $Q \in \Fc_{\kappa}$ such that $\KS(Q,\F_n) \le \eta_n + \epsilon$. Hence, on the event $A_n$, it holds that $\KS(Q,F^*) \le 2\eta_n + \epsilon$. Since $\eta_n>0$, taking $\epsilon=\eta_n$ gives $J(F^*;3\eta_n) \le \kappa$. Taking the infimum over all such $\kappa$ such that $\delta_n(\kappa) \le \eta_n$ yields
\begin{align*}
    J(F^*;3\eta_n) \le \hat\kappa_0.
\end{align*}
by the definition of $\hat\kappa_0$ in \eqref{eq:poisson neighborhood procedure}. Since the infimum in \eqref{eq:kappa0} is attained, we have $F^* \in \Fc_{\kappa_0}$ by \eqref{eq:kappa0 and observable distribution}. Hence, on the event $A_n$, it holds that
\begin{align*}
    \delta_n(\kappa_0) \le \KS(F^*,\F_n) \le \eta_n,
\end{align*}
which implies $\hat\kappa_0 \le \kappa_0$. Therefore, using $J(F^*)=\kappa_0$ by \eqref{eq:kappa0 and observable distribution}, we obtain
\begin{align*}
    J(F^*;3\eta_n) \le \hat\kappa_0 \le \kappa_0
\end{align*}
for all sufficiently large $n$ almost surely. Since $3\eta_n \to 0$ and the lower envelope $J(F^*;3 \eta_n)$ in \eqref{eq:J envelope} converges to $J(F^*)$ from below by the lower semi-continuity of $J$ at $F^*$ from (i), we have $\hat\kappa_0 \to \kappa_0$ almost surely. 

(iii) Finally, if $\eta_n = C\sqrt{(\log n)/n}$ with $C > 1/\sqrt{2}$, then the Dvoretzky-Kiefer-Wolfowitz (DKW) inequality with the sharp constant of \citet{Massart1990} gives
\begin{align*}
    \P_{\trueprior}(A_n^c) \le 2e^{-2n\eta_n^2} = 2n^{-2C^2}.
\end{align*}
Since $2C^2 > 1$, the Borel-Cantelli lemma now implies that $A_n$ occurs eventually almost surely, and the conclusion follows.
\end{proof}
\end{document}